\documentclass[a4paper,12pt]{article}

\usepackage{amsmath,amssymb,amsfonts,amsthm}
\usepackage{multirow,bigdelim}
\usepackage{amscd}
\usepackage{verbatim}
\usepackage{latexsym}
\usepackage{array}
\usepackage{enumerate}
\usepackage[all]{xy}
\usepackage{graphicx}
\usepackage{ulem}
\usepackage{ascmac}
\usepackage{mathtools}
\usepackage{layout}

\usepackage[initials,alphabetic]{amsrefs}
   \BibSpec{arXiv}{
   +{}{\PrintAuthors}{author}
   +{,}{ \textit}{title}
   +{,}{ }{date}
   +{,}{ arXiv:}{eprint}}
\usepackage{mleftright}
   \mleftright

\usepackage{tikz}
   \usetikzlibrary{cd,decorations,decorations.pathreplacing,arrows,calc,intersections,hobby}
\usepackage{tikz-cd}

\numberwithin{equation}{section}

\newcommand{\CC}{\mathbb{C}}

\newcommand{\RR}{\mathbb{R}}

\newcommand{\ZZ}{\mathbb{Z}}
\newcommand{\PP}{\mathbb{P}}
\newcommand{\EE}{\mathbb{E}}

\newcommand{\LL}{\mathbb{L}}

\newcommand{\calS}{\mathcal{S}}

\newcommand{\SP}{\mathcal{P}}

\newcommand{\SD}{\mathcal{D}}
\newcommand{\SO}{\mathcal{O}}

\newcommand{\SM}{\mathcal{M}}
\newcommand{\SN}{\mathcal{N}}
\newcommand{\SL}{\mathcal{L}}
\newcommand{\SE}{\mathcal{E}}
\newcommand{\SF}{\mathcal{F}}

\newcommand{\bfD}{\mathbf{D}}
\newcommand{\bfQ}{\mathbf{Q}}
\newcommand{\bfE}{\mathbf{E}}

\newcommand{\sfE}{\mathsf{E}}
\newcommand{\sfL}{\mathsf{L}}

\newcommand{\rmI}{\mathrm{I}}
\newcommand{\rmb}{\mathrm{b}}
\newcommand{\rmR}{{\mathrm{R}}}
\newcommand{\rmE}{{\mathrm{E}}}

\newcommand{\rmc}{\mathrm{c}}

\newcommand{\op}{\mathrm{op}}

\newcommand{\an}{\mathrm{an}}

\newcommand{\id}{\mathrm{id}}

\DeclareMathOperator{\Ker}{Ker}

\DeclareMathOperator{\rank}{rk}

\DeclareMathOperator{\supp}{supp}

\DeclareMathOperator{\Char}{char}

\DeclareMathOperator{\GL}{GL}

\DeclareMathOperator{\CF}{CF}

\renewcommand{\Re}{\operatorname{Re}}

\newcommand{\Dotimes}{\overset{D}{\otimes}}

\newcommand{\Potimes}{\overset{+}{\otimes}}

\newcommand{\rhom}{{\rm R}{\mathcal{H}}om}

\newcommand{\shom}{{\mathcal{H}}om}

\newcommand{\rsect}{{\mathrm{R}}\Gamma}

\newcommand{\SolE}{Sol^{\mathrm{E}}}

\newcommand{\ssupp}{\operatorname{sing}\!.\!\operatorname{supp}}

\newcommand{\BDC}{{\mathbf{D}}^{\mathrm{b}}}

\newcommand{\Mod}{\mathrm{Mod}}
\newcommand{\Db}{{\bf D}^{\mathrm{b}}}
\newcommand{\Dbc}{{\bf D}_{\mathrm{c}}^{\mathrm{b}}}

\newcommand{\Modcoh}{\mathrm{Mod}_{\mathrm{coh}}}
\newcommand{\Modhol}{\mathrm{Mod}_{\mathrm{hol}}}
\newcommand{\Modrh}{\mathrm{Mod}_{\mathrm{rh}}}

\newcommand{\BDCcoh}{{\mathbf{D}}^{\mathrm{b}}_{\mathrm{coh}}}

\newcommand{\BDChol}{{\mathbf{D}}^{\mathrm{b}}_{\mathrm{hol}}}
\newcommand{\BDCrh}{{\mathbf{D}}^{\mathrm{b}}_{\mathrm{rh}}}
\newcommand{\BEC}{{\mathbf{E}}^{\mathrm{b}}}

\newcommand{\ECrc}{\mathbf{E}_{\mathbb{R}{\text -}\mathrm c}}
\newcommand{\BECrc}{\mathbf{E}_{\mathbb{R}{\text -}\mathrm c}^{\mathrm{b}}}

\renewcommand{\(}{\left(}
\renewcommand{\)}{\right)}
\newcommand{\bs}{\left.\right\backslash} 
\newcommand{\vbar}{\left.\right|} 
\DeclareRobustCommand{\longtwoheadrightarrow}{\relbar\joinrel\twoheadrightarrow} 
\DeclarePairedDelimiter{\abs}{\lvert}{\rvert} 
\DeclarePairedDelimiter{\innpro}{\langle}{\rangle} 
\DeclarePairedDelimiterX{\Set}[2]{\lbrace}{\rbrace}{#1\ \delimsize\vert\ #2} 
\newcommand{\simto}{\overset{\sim}{\longrightarrow}}

\newcommand{\longhookrightarrow}{\lhook\joinrel\longrightarrow}

\newcommand{\var}[1]{\overline{#1}}
\newcommand{\tl}[1]{\widetilde{#1}}

\newcommand{\inj}{``\varinjlim"} 
\newcommand{\Modleq}{\mathrm{Mod}_{\mathrm{hol}}^{\leq1}} 
\newcommand{\CCE}{\mathbb{C}^\mathrm{E}} 

\newcommand{\F}{\mathcal{F}}
\newcommand{\G}{\mathcal{G}}

\renewcommand{\(}{\left(}
\renewcommand{\)}{\right)}

\renewcommand{\dim}{{\rm dim}}

\newcommand{\dsum}{\displaystyle \sum}
\newcommand{\dint}{\displaystyle \int}
\newcommand{\dprod}{\displaystyle \prod}

\newtheorem{theorem}{Theorem}[section]

\newtheorem{corollary}[theorem]{Corollary}
\newtheorem{lemma}[theorem]{Lemma}
\newtheorem{proposition}[theorem]{Proposition}

\theoremstyle{definition}
\newtheorem{definition}[theorem]{Definition}

\theoremstyle{remark}
\newtheorem{remark}[theorem]{\sc Remark}

\newtheorem{example}[theorem]{\sc Example}

\title{On the monodromies at infinity of Fourier transforms
of holonomic {$\SD$}-modules
}

\author{Kazuki KUDOMI
\footnote{Mathematical Institute, Tohoku University,
Aramaki Aza-Aoba 6-3, Aobaku, Sendai, 980-8578, Japan.
E-mail: kazuki.kudomi.q3@dc.tohoku.ac.jp}
and Kiyoshi TAKEUCHI
\footnote{Mathematical Institute, Tohoku University,
Aramaki Aza-Aoba 6-3, Aobaku, Sendai, 980-8578, Japan.
E-mail: takemicro@nifty.com} }

\begin{document}

\maketitle

\begin{abstract}
Based on the recent progress in the irregular
Riemann-Hilbert correspondence,
we study the monodromies at infinity of
the holomorphic solutions of
Fourier transforms of holonomic
D-modules in some situations.
Formulas for their eigenvalues
are obtained by applying the theory of
monodromy zeta functions to
our previous results on the enhanced
solution complexes of the Fourier transforms.
In particular, in dimension one we thus find
a reciprocity law between
the monodromies at infinity of holonomic D-modules
and their Fourier transforms.
\end{abstract}

\section{Introduction}\label{intro}

First of all 
we recall the definition of 
Fourier transforms of algebraic $\SD$-modules. 
Let $X=\CC_z^N$ be a complex vector space 
and $Y=\CC_w^N$ its dual. 
We regard them as algebraic varieties and 
use the notations $\SD_X$ and $\SD_Y$ for 
the sheaves of rings of ``algebraic'' differential operators on them. 
Denote by $\Modcoh(\SD_X)$ (resp. $\Modhol(\SD_X)$, $\Modrh(\SD_X)$) 
the category of coherent (resp. holonomic, regular holonomic) 
$\SD_X$-modules on $X$. 
Let $W_N := \CC[z, \partial_z]\simeq\Gamma(X; \SD_X)$ and 
$W^\ast_N := \CC[w, \partial_w]\simeq\Gamma(Y; \SD_Y)$ 
be the Weyl algebras over $X$ and $Y$, respectively. 
Then by the ring isomorphism 
\begin{equation*}
W_N\simto W^\ast_N\hspace{30pt}
(z_i\mapsto-\partial_{w_i},\ \partial_{z_i}\mapsto w_i) 
\end{equation*}
we can endow a left $W_N$-module $M$ with 
a structure of a left $W_N^\ast$-module.
We call it the Fourier transform of $M$ 
and denote it by $M^\wedge$. 
For a ring $R$ we denote by $\Mod_f(R)$ 
the category of finitely generated $R$-modules. 
Recall that for the affine algebraic varieties $X$ 
and $Y$ we have the equivalences of categories 
\begin{align*} 
\Modcoh(\SD_X)
&\simeq
\Mod_f(\Gamma(X; \SD_X)) = \Mod_f(W_N),\\
\Modcoh(\SD_Y)
&\simeq
\Mod_f(\Gamma(Y; \SD_Y)) = \Mod_f(W^\ast_N)
\end{align*}
(see e.g. \cite[Propositions 1.4.4 and 1.4.13]{HTT08}).
For a coherent $\SD_X$-module $\SM\in\Modcoh(\SD_X)$
we thus can define its Fourier 
transform $\SM^\wedge\in\Modcoh(\SD_Y)$. 
For $m,n>$ let $M(m,n,W_N)$ (resp. $M(m,n,W_N^*)$) be the 
set of $m \times n$ matrices with entries in 
$W_N$ (resp. $W_N^*$). Then by the Fourier transform 
the system $P \vec{u} = \vec{0}$ of partial 
differential equations for a matrix 
$P=(P_{ij}) \in M(m,n,W_N)$ on $X$ corresponds to the one 
$P^{\wedge} \vec{v} = \vec{0}$ for 
$P^{\wedge}=(P_{ij}^{\wedge}) \in M(m,n,W_N^*)$ on $Y$.
There exists also an equivalence of categories
\begin{equation*}
( \cdot )^\wedge : \Modhol(\SD_X)\simto \Modhol(\SD_Y)
\end{equation*}
between the categories of holonomic $\SD$-modules 
(see e.g. \cite[Proposition 3.2.7]{HTT08}).

Although this definition of Fourier transforms of $\SD$-modules 
is simple, it is very hard to describe their basic 
properties. In particular, we do not 
know so far much of their global properties, 
such as the monodromies and the Stokes
matrices of their holomorphic solutions, even in
the simplest case $N=1$ of dimension one (see
for example, D'Agnolo-Hien-Morando-Sabbah
\cite{DHMS20} and Hohl \cite{Hoh22}). 
For the results in higher dimensions, see
Brylinski \cite{Bry86}, Daia \cite{Dai00},
Ito-Takeuchi \cite{IT20a}, \cite{IT20b}, 
Kashiwara-Schapira \cite {KS97} 
and Takeuchi \cite{Tak22}. After some
pioneering works by Malgrange in \cite{Mal88} and
\cite{Mal91}, inspired from the theory of
Fourier transforms of $l$-adic sheaves in positive
characteristic, Bloch-Esnault \cite{BE04b} and
Garc{\'i}a L{\'o}pez \cite{Gar04} introduced
the so-called local Fourier transforms of
algebraic holonomic
$\SD$-modules $\SM$ on the affine line $X= \CC_z$.
Then in the one dimensional case $N=1$ their new
method enables us to describe the formal structure
i.e. the exponential factors of the Fourier transform
$\SM^\wedge$ of $\SM$ in terms of that of $\SM$ (for 
an explicit description, see Fang \cite{Fan09}, 
Graham-Squire \cite{Gra13} 
and Sabbah \cite{Sab08}). Subsequently, based on this
result, Mochizuki \cite{Mochi10}, \cite{Mochi18} gave also
a description of the Stokes structure of $\SM^\wedge$.
We thus now know that the exponential factors
of $\SM^\wedge$ are obtained by the Legendre
transform from those of $\SM$. We call it
the stationary phase method. Later, by using
the new theories of the irregular Riemann-Hilbert
correspondence established by D'Agnolo-Kashiwara
\cite{DK16} and the enhanced Fourier-Sato transforms
of Kashiwara-Schapira \cite{KS16a} adapted to it, in \cite{DK18} and \cite{DK23}
D'Agnolo and Kashiwara reformulated and reproved it more elegantly. 
See also Kudomi-Takeuchi \cite{KT23} for a recent 
improvement of their results and some new applications. 

In this paper, we study the monodromies at infinity of 
the holomorphic solutions of Fourier transforms of 
holonomic $\SD$-modules. First in the case $N=1$,  
we prove a reciprocity law between 
the monodromies at infinity of algebraic holonomic $\SD$-modules on 
the affine line $X=\CC_z$ and their Fourier transforms. 
For an algebraic holonomic $\SD$-modules $\SM$ on $X=\CC_z$ 
we take a large enough $R\gg0$ such that 
\begin{equation}
\ssupp(\SM)\subset\{z\in X^{\an} =\CC\mid\abs{z}<R\}, 
\end{equation} 
where $X^{\an} \simeq \CC$ stands for 
the underlying complex manifold of $X$. 
Then the solution complex $Sol_X(\SM)\in\Dbc(\CC_{X^\an})$ of $\SM$ 
is a local system on the large circle 
\begin{equation}
C_R\coloneq\{z\in X^{\an} =\CC\mid\abs{z}=R\}
\quad \subset X^{\an} =\CC
\end{equation} 
and we can define the monodromy at infinity of 
(the holomorphic solutions) to $\SM$
as follows.
Take a point $a\in C_R$ of $C_R$ and let $\Psi_a$ be 
the $\CC$-linear automorphism of the stalk 
\begin{equation}
Sol_X(\SM)_a\simeq
\shom_{\SD_{X^\an}}(\SM^\an,\SO_{X^\an})_a
\end{equation}
at $a$ induced by the analytic continuations along the loop $C_R$ 
in the ``counterclockwise" direction. 
We denote by $\rank\SM$ the generic rank of 
the holonomic $\SD_X$-modules $\SM$ so that we have 
$Sol_X(\SM)_a\simeq\CC^{\rank{\SM}}$.
Then the square matrix $A$ of size $\rank\SM$ representing 
the automorphism $\Psi_a$ is defined up to conjugacy. 
We call it the monodromy (matrix) at infinity of $\SM$.
It is clear that the conjugacy class of $A$ in 
$\GL_{\rank\SM}(\CC)$ does not depend on 
the choices of $R\gg0$ and $a\in C_R$.
Similarly, also for the Fourier transform $\SM^\wedge$ of 
$\SM$ we define its monodromy at infinity.
In order to relate the monodromy 
at infinity of $\SM$ to that of $\SM^\wedge$, 
we require the following condition at infinity.
Let $\var{X}\simeq\PP^1$ be 
the projective compactification of $X=\CC$ and 
$i_X\colon X\hookrightarrow\var{X}$ the inclusion map.
We set $\tl{\SM}\coloneq\bfD i_{X\ast}\SM\simeq i_{X\ast}\SM
\in\Modhol(\SD_{\var{X}})$.
Note that $\tl{\SM}$ is an algebraic meromorphic connection 
on a (Zariski) open neighborhood of the point 
$\infty\in\var{X}=\PP^1$.

\begin{definition}
We say that the algebraic holonomic $\SD_X$-module $\SM$ 
on $X=\CC$ has a moderate irregularity at infinity 
if the pole order of any exponential factor of 
$\tl{\SM}$ at $\infty\in\var{X}=\PP^1$ is $\leq1$. 
\end{definition}
We denote by $\Modleq(\SD_X)$ the full subcategory of 
$\Modhol(\SD_X)$ consisting of holonomic $\SD_X$-modules 
with a moderate irregularity at infinity.
Then by the multiplicity test functor 
in \cite[Section 6.3]{DK18} we see that 
it is an abelian subcategory of $\Modhol(\SD_X)$.
It is clear that $\Modrh(\SD_X)$ is a subcategory of $\Modleq(\SD_X)$.
Moreover by the classical stationary phase method, 
we can easily show that $\SM\in\Modhol(\SD_X)$ has 
a moderate irregularity at infinity 
if and only if $\SM^\wedge\in\Modhol(\SD_Y)$ has it. 
Thus it turns out that the Fourier transform $( \cdot )^\wedge$ induces 
an equivalence of categories 
\begin{equation}
( \cdot )^\wedge\colon\Modleq(\SD_X)\overset{\sim}{\longrightarrow}
\Modleq(\SD_Y).
\end{equation}
For a non-zero complex number $\lambda\in\CC^\ast=\CC\bs\{0\}$ 
we denote by $\mu (\SM,\lambda)$ (resp. $\mu (\SM^\wedge,\lambda))$ 
the multiplicity of the eigenvalue 
$\lambda$ in the monodromy at infinity of $\SM$ (resp. $\SM^\wedge$).
Then our first main theorem in this paper is 
the following reciprocity law between the monodromies at infinity of 
$\SM$ and $\SM^\wedge$.

\begin{theorem}\label{MT-1} 
Assume that the holonomic $\SD_X$-module $\SM\in\Modhol(\SD_X)$ has 
a moderate irregularity at infinity i.e. $\SM \in \Modleq(\SD_X)$.  
Then for any non-zero complex number 
$\lambda\in\CC^\ast=\CC\bs\{0\}$ such that $\lambda\neq1$ we have an equality
\begin{equation}
\mu (\SM,\lambda)= \mu (\SM^\wedge,\lambda^{-1}).
\end{equation}
\end{theorem}
In addition to the use of the machineries 
in \cite{DK16} and \cite{KS16a}, for the 
proof of Theorem \ref{MT-1} we apply the theory of monodromy 
zeta functions in singularity theory (see Section \ref{sec:zeta} for a 
review of it)  
to the precise description  
of the enhanced solution complex of $\SM^{\wedge}$ obtained in 
\cite{KT23}. It seems that such a result can be obtained 
most neatly through the theories of 
the irregular Riemann-Hilbert 
correspondence. 
In Section \ref{dim-N}, we partially extend Theorem \ref{MT-1} to the 
higher-dimensional case $N \geq 1$ and show that the 
eigenvalues $\lambda \neq 1$ of the 
monodromies at infinity of Fourier transforms of 
regular holonomic $\SD$-modules on $X= \CC^N$ can be calculated 
explicitly. Recall that by \cite[Corollary 4.5]{IT20a} for 
a regular holonomic $\SD$-module $\SM$ on $X= \CC^N$ 
we obtain a $\CC^*$-conic open subset 
$\Omega \not= \emptyset$ of $Y= \CC^N$ such that 
the restriction of $Sol_Y( \SM^\wedge )$ to it 
is a local system. Then we define the monodromy 
at infinity of the Fourier transform $\SM^\wedge$ to 
be the monodromy of the local system along a large 
circle in a $\CC^*$-orbit in $\Omega$. 
For the precise definition, see Section \ref{dim-N}. 
Also, for similar results on a slightly different type of 
monodromies at infinity of Fourier transforms of 
some special regular holonomic $\SD$-modules, see 
\cite{AET15} and \cite{Tak10}.

\section{Preliminaries}\label{prelim}

In this section, we recall some basic notions 
and results which will be used in this paper.
We assume here that the reader is familiar with 
the theory of sheaves and functors in the 
framework of derived categories.
For them we follow the terminologies in \cite{KS90} etc.
For a topological space $X$ denote by $\BDC(\CC_X)$  
the derived category consisting of bounded 
complexes of sheaves of $\CC$-vector spaces on it. If 
$X$ is a complex manifold, we denote by $\Dbc ( \CC_X)$ 
the full subcategory of $\BDC(\CC_X)$ 
consisting of objects with constructible cohomologies.

\subsection{Enhanced sheaves}\label{sec:enh}
We refer to \cite{DK16} and \cite{KS16b} 
for the details of this subsection.
Let $X$ be a complex manifold and 
we consider the maps 
\begin{align}
X\times\RR^2 \xrightarrow{p_1,p_2,\mu} 
X\times\RR \overset{\pi}{\longrightarrow} X,
\end{align}
where $p_1,p_2,\pi$ are the projections and we set 
$\mu(x,t_1,t_2)\coloneq(x,t_1+t_2)$.
Then we define the bounded derived 
category of enhanced sheaves $\BEC(\CC_X)$ on $X$ by
\begin{align}
\BEC(\CC_X) \coloneq \BDC(\CC_{X\times\RR})/\pi^{-1}\BDC(\CC_X).
\end{align}
The convolution functor $\Potimes$ in 
$\BDC(\CC_{X\times\RR})$ is defined by 
\begin{equation}
F\Potimes G \coloneq \rmR\mu_!(p_1^{-1}F\otimes p_2^{-1}G) 
\end{equation} 
and it induces a convolution functor in $\BEC(\CC_X)$, 
which we denote also by $\Potimes$. In this situation, 
there exists a quotient functor 
\begin{align}
\bfQ \colon \BDC(\CC_{X\times\RR})\longrightarrow\BEC(\CC_X). 
\end{align}
For a morphism of complex manifolds $f\colon X\rightarrow Y$,
we define the direct image functor $\bfE f_\ast\colon 
\BEC(\CC_X)\rightarrow\BEC(\CC_Y)$
so that the following diagram commutes:
\begin{equation}
\vcenter{
\xymatrix@M=7pt@C=36pt@R=24pt{
\BDC(\CC_{X\times\RR}) \ar[r]^-{\rmR(f\times\id_{\RR})_\ast} 
\ar[d]^-{\bfQ} & \BDC(\CC_{Y\times\RR}) \ar[d]^-{\bfQ} & \\
\BEC(\CC_X) \ar[r]^-{\bfE f_\ast} & \BEC(\CC_Y). &
}}
\end{equation} 
The proper direct image and (proper) inverse image functors
\begin{align}
\bfE f_! &\colon 
\BEC(\CC_X)\longrightarrow\BEC(\CC_Y), \\
\bfE f^{-1},\bfE f^! &\colon 
\BEC(\CC_Y)\longrightarrow\BEC(\CC_X)
\end{align}
are defined similarly. We have a natural embedding 
$\epsilon\colon\BDC(\CC_X)\rightarrow\BEC(\CC_X)$ defined by
\begin{align}
\epsilon(L) \coloneq \bfQ(\CC_{\{t\geq0\}}\otimes\pi^{-1}L),
\end{align}
where $\{t\geq0\}$ stands for 
$\Set*{\(x,t\)\in X\times\RR}{t\geq0}$.

\subsection{Enhanced ind-sheaves}\label{sec:enhind}

We briefly recall some notions and results on 
enhanced ind-sheaves without giving detailed definitions.
We refer to \cite{KS01} for ind-sheaves, 
to \cite{DK16} for ind-sheaves on bordered spaces, and 
to \cite{DK16} and \cite{KS16b} for enhanced ind-sheaves.
Let $X$ be a complex manifold and $\RR_\infty$ 
the bordered space $(\RR, \var{\RR}\coloneq\RR\sqcup\{\pm\infty\})$.
Denote by $\BDC(\rmI\CC_X)$ and $\BDC(\rmI\CC_{X\times\RR_\infty})$ 
the bounded derived categories of ind-sheaves on $X$ and 
ind-sheaves on the bordered space $X\times\RR_\infty$, respectively.
We define the the bounded derived 
category of enhanced ind-sheaves $\BEC(\rmI\CC_X)$ on $X$ by 
\begin{equation}
\BEC(\rmI\CC_X) \coloneq 
\BDC(\rmI\CC_{X\times\RR_\infty})/\pi^{-1}\BDC(\rmI\CC_X),
\end{equation}
where $\pi\colon X\times\RR_\infty\rightarrow X$ 
is the projection of bordered spaces.
Then there exists a quotient functor 
\begin{align}
\bfQ \colon \BDC(\rmI\CC_{X\times\RR_\infty})
\longrightarrow\BEC(\rmI\CC_X). 
\end{align}
Furthermore, as in the case of enhanced sheaves 
we can define the convolution functor
\begin{align}
\Potimes \colon \BEC(\rmI\CC_X)\times\BEC(\rmI\CC_X)
\longrightarrow\BEC(\rmI\CC_X)
\end{align}
and the operations of (proper) direct and inverse images
\begin{align}
\bfE f_\ast,\bfE f_{!!} &\colon \BEC(\rmI\CC_X)
\longrightarrow\BEC(\rmI\CC_Y), \\
\bfE f^{-1},\bfE f^! &\colon \BEC(\rmI\CC_Y)
\longrightarrow\BEC(\rmI\CC_X)   
\end{align} 
for a morphism of complex manifolds $f\colon X\rightarrow Y$. 
We have a natural embedding $\epsilon\colon\BDC(\rmI\CC_X)
\rightarrow\BEC(\rmI\CC_X)$ defined by
\begin{align}
\epsilon(\SL) \coloneq 
\bfQ(\CC_{\{t\geq0\}}\otimes\pi^{-1}\SL). 
\end{align}
We set $\CC_X^\rmE\coloneq\bfQ\Bigl(\underset{\alpha \to+\infty}
{\inj}\CC_{\{t\geq \alpha \}}\Bigr)\in\BEC(\rmI\CC_X)$, where 
the symbol $\inj$ stands for the inductive limit  
in $\BDC(\rmI\CC_{X\times\var{\RR}})$.

\subsection{Exponential enhanced (ind-)sheaves}\label{sec:exp-enhind}

Let $X$ be a complex manifold.
Denote by $\BECrc(\CC_X)$ and $\BECrc(\rmI\CC_X)$ 
the triangulated categories of $\RR$-constructible 
enhanced (ind-)sheaves on $X$
(see \cite{DK16}).
Let $U\subset X$ be an open subset and $\phi \colon U\to\RR$ 
a continuous function on it. 
For a locally closed subset $Z\subset U$, we define the exponential 
enhanced (ind-)sheaves $\sfE_{Z\vbar X}^\phi \in \BEC (\CC_X)$ 
and $\EE_{Z\vbar X}^\phi \in \BEC (\rmI\CC_X)$ by 
\begin{align}
\sfE_{Z\vbar X}^\phi \coloneq \bfQ(\CC_{\{t+\phi\geq0\}}), \quad 
\EE_{Z\vbar X}^\phi \coloneq \CC_X^\rmE\Potimes\sfE_{Z\vbar X}^\phi, 
\end{align}
where $\{t+\phi\geq0\}$ stands for 
$\Set*{(x,t)\in X\times\RR}{x \in Z, t+\phi(x)\geq0}$.

\subsection{$\SD$-modules}\label{sec:Dmod}

Let us recall some notions and results on 
$\SD$-modules on a complex manifold $X$
(we refer to \cite{Kas03} and \cite{HTT08} etc.).
Denote by $\SO_X,\Omega_X$ and $\SD_X$ the sheaves 
of holomorphic functions, holomorphic differential forms 
of top degree and holomorphic differential operators on $X$, respectively.
Let $\Mod(\SD_X)$ be the abelian category of left $\SD_X$-modules.
Then we can define $\Modcoh(\SD_X)$ (resp. $\Modhol(\SD_X), \Modrh(\SD_X)$) 
to be the subcategory of $\Mod(\SD_X)$ consisting of 
coherent (resp. holonomic, regular holonomic) $\SD_X$-modules.
We write $\BDC(\SD_X)$ for the bounded derived category of 
left $\SD_X$-modules and denote by $\BDCcoh(\SD_X)$,  
$\BDChol(\SD_X)$ and $\BDCrh(\SD_X)$ its full triangulated subcategories of 
objects which have coherent,
holonomic and regular holonomic cohomologies  respectively.
The symbols $\Dotimes,\bfD f_\ast,\bfD f^\ast$ stand for 
the standard operations for $\SD$-modules associated to 
a morphism of complex manifolds $f\colon X\to Y$.
The solution functor is defined by
\begin{align}
Sol_X \colon \BDCcoh(\SD_X)^{\op}\longrightarrow
\BDC(\CC_X), \qquad \SM\longmapsto\rhom_{\SD_X}(\SM,\SO_X). 
\end{align}
Let $D\subset X$ be a closed hypersurface and 
denote by $\SO(\ast D)$ the sheaf of 
meromorphic functions on $X$ with poles in $D$.
Then for $\SM\in\BDC(\SD_X)$, we set 
\begin{align}
\SM(\ast D)\coloneq\SM\Dotimes\SO_X(\ast D)
\end{align}
and for $f\in\SO_X(\ast D)$ and $U\coloneq X\bs D$, set
\begin{align}
\SD_X e^f &\coloneq \SD_X / 
\Set{P\in\SD_X}{P e^f\vbar_U=0}, \\
\SE_{U\vbar X}^f &\coloneq \SD_X e^f(\ast D).
\end{align}
Note that $\SE_{U\vbar X}^f$ is holonomic.

In \cite{DK16}, D'Agnolo and Kashiwara constructed 
the enhanced solution functor on a complex manifold $X$
\begin{align}
Sol_X^\rmE\colon\BDChol(\SD_X)^{\op}
\longrightarrow\BECrc(\rmI\CC_X)
\end{align}
and proved that it is fully faithful.
Instead of giving its definition, we recall some 
of their properties.

\begin{proposition}\label{prop-K2}
Let $D\subset X$ be a closed hypersurface in 
$X$ and set $U\coloneq X\bs D$.
\begin{enumerate}
\item[{\rm(i)}] If $\SM\in\BDChol(\SD_X)$, 
then there exists an isomorphism in $\BEC(\rmI\CC_X)$
\begin{align}Sol_X^{\rmE}(\SM(\ast D)) \simeq 
\pi^{-1}\CC_U\otimes Sol_X^{\rmE}(\SM).
\end{align}
\item[{\rm(ii)}] Let $f\colon X\to Y$ 
be a morphism of complex manifolds.
If $\SN\in\BDChol(\SD_Y)$, then there 
exists an isomorphism in $\BEC(\rmI\CC_X)$
\begin{equation}
Sol_X^{\rmE}(\bfD f^\ast\SN)\simeq\bfE f^{-1}Sol_Y^{\rmE}(\SN).
\end{equation}
\item[{\rm(iii)}] Let $f\colon X\to Y$ 
be a morphism of complex manifolds.
If $\SM\in\BDChol(\SD_X)$ and $\supp(\SM)$ is proper over $Y$, then there 
exists an isomorphism in $\BEC(\rmI\CC_Y)$
\begin{equation}
Sol_X^{\rmE}(\bfD f_\ast\SM)[d_Y]\simeq\bfE f_\ast Sol_Y^{\rmE}(\SM)[d_X]
\end{equation}
where $d_X$ (resp. $d_Y$) is the complex dimension of $X$ (resp. $Y$).
\item[{\rm(iv)}] If $f\in\SO_X(\ast D)$, 
then there exists an isomorphism in $\BEC(\rmI\CC_X)$
\begin{align}
Sol_X^{\rmE}(\SE_{U\vbar X}^f) 
\simeq \EE_{U\vbar X}^{\Re f}.
\end{align}
\item[{\rm(v)}] Let $\SM$ be 
a regular holonomic $\SD_X$-module and 
set $L\coloneq Sol_X(\SM)$.
Then we have an isomorphism in $\BEC(\rmI\CC_X)$
\begin{align}
Sol_X^{\rmE}(\SM)\simeq\CC_X^{\rmE}\Potimes\epsilon(L).
\end{align}
\end{enumerate}
\end{proposition}

\subsection{Puiseux germs and normal forms of 
enhanced (ind-)sheaves}\label{sec:K6}

In this subsection, we recall some definitions on 
Puiseux germs and 
normal forms of enhanced (ind-)sheaves 
in \cite[Section 5]{DK18} to describe 
the Hukuhara-Levelt-Turrittin theorem 
in terms of enhanced ind-sheaves.
Let $X$ be a complex manifold of dimension one.
For $a\in X$, denote by $\varpi_a\colon \tl{X_a}\to X$ 
the real oriented blow-up of $X$ along $a$ and consider 
the commutative diagram
\begin{equation}
\vcenter{
\xymatrix@M=5pt{
S_aX \ar[r]^-{\tl{\imath}_a} & \tl{X_a} \ar[d]^-{\varpi_a}& \\
X\bs\{a\} \ar[r]^-{j_a} \ar[ur]^-{\tl{\jmath}_a} & X, 
}}\end{equation}
where we set $S_aX\coloneq\varpi_a^{-1}(a)\simeq 
S^1$ and $\tl{\imath}_a,\tl{\jmath}_a,j_a$ 
are the natural embeddings.
We denote by $\SP_{\tl{X_a}}$ the subsheaf of 
$\tl{\jmath}_{a\ast} j_a^{-1}\SO_X$ whose sections are defined by 
\begin{multline*}
\Gamma(\Omega;\SP_{\tl{X_a}}) \coloneq 
\{f\in\Gamma(\Omega;\tl{\jmath}_{a\ast} 
j_a^{-1}\SO_X)\mid \textit{For any $\theta\in\Omega\cap S_aX$,} \\ 
\textit{$f$ admits a Puiseux expansion at $\theta$.}\}
\end{multline*}
for open subsets $\Omega\subset\tl{X_a}$.
Then we define the sheaf of 
Puiseux germs $\SP_{S_aX}$ on $S_aX$ to be  
\begin{align}
\SP_{S_aX}\coloneq\tl{\imath}_a^{\,-1}\SP_{\tl{X_a}}.
\end{align}
Let $z_a$ be a local coordinate centered at $a$ and  
denote by $\SP_{S_aX}^\prime$ the subsheaf of 
$\SP_{S_aX}$ consisting of sections locally contained in
\begin{align}
\bigcup_{p\in\ZZ_{\geq1}}z_a^
{-\frac{1}{p}}\CC[z_a^{-\frac{1}{p}}]
\end{align}
for some (hence, any) branch of $z_a^{1/p}$.
Let $\SM$ be a holonomic $\SD_X$-module. For $a\in X$
the enhanced solution complex $Sol_X^{\rmE}(\SM)$ has 
the following decomposition by some 
exponential enhanced ind-sheaves on 
sufficiently small open sectors along $a$.

\begin{lemma}[see e.g. {\cite[Corollary 3.7]{IT20a}}]\label{lem-ITa}
Let $\SM$ be a holonomic $\SD_X$-module and let $a\in X$.
Then for any $\theta\in S_aX$, 
there exist its sectorial neighborhood 
$V_\theta \subset X \setminus \{ a \} \subset \tl{X_a}$ 
and holomorphic functions 
$f_1,\dots,f_m\in\Gamma(V_\theta;\SP_{\tl{X_a}})$ such that
\begin{align}
\pi^{-1}\CC_{V_\theta}\otimes Sol_X^{\rmE}(\SM) \simeq 
\bigoplus_{i=1}^m\EE_{V_\theta\vbar X}^{\Re f_i}.
\end{align}
\end{lemma}
In \cite{DK18}, D'Agnolo and Kashiwara introduced some 
notions on normal forms of enhanced sheaves to refine the above decomposition.
We recall their  definitions in a slightly modified form as follows.

\begin{definition}
Let $a\in X$ and let 
$N\colon\SP_{S_aX}^\prime\to(\ZZ_{\geq0})_{S_aX}$ be 
a morphism of sheaves of sets.
Then the morphism $N$ is said to be a multiplicity at $a$ if 
$N_\theta^{>0}\coloneq N_\theta^{-1}(\ZZ_{>0})
\subset\SP_{S_aX,\theta}^\prime$ is a finite set  
for any $\theta\in S_aX$.
\end{definition}

\begin{definition}
Let $\SF\in\ECrc^{\rmb}(\rmI\CC_X)$ be an $\RR$-constructible 
enhanced ind-sheaf and $N\colon\SP_{S_aX}^\prime\to(\ZZ_{\geq0})_{S_aX}$ 
a multiplicity at $a\in X$.
Then we say that $\SF$ has 
a normal form at $a\in X$ for $N$ if any $\theta\in S_aX$ has 
a sectorial open neighborhood $V_\theta \subset 
X \setminus \{ a \} \subset \tl{X_a}$ for which we have 
an isomorphism
\begin{align}
\pi^{-1}\CC_{V_\theta}\otimes \SF\simeq 
\bigoplus_{f\in N_\theta^{>0}}
(\EE_{V_\theta\vbar X}^{\Re f})^{N_\theta(f)}.
\end{align}
\end{definition}
Enhanced sheaves on $X$ 
having a normal form at $a\in X$ are defined similarly. 
Note that if an enhanced ind-sheaf  
$\SF\in\ECrc^{\rmb}(\rmI\CC_X)$ has 
a normal form at $a\in X$, then the multiplicity $N$
for which $\SF$ has a normal form
at $a$ is uniquely determined (see {\cite[Lemma 7.15]{Mochi22}}, 
{\cite[Corollary 5.2.3]{DK18}} and {\cite[Proposition 3.16]{IT20a}}). 
Then we obtain the following proposition.

\begin{proposition}[{\cite[Lemma 5.4.4]{DK18}}]\label{prop-K4}
Let $\SM$ be a holonomic $\SD_X$-module.
Then for any point $a\in X$ the enhanced solution complex 
$Sol_X^{\rmE}(\SM)\in\ECrc^{\rmb}(\rmI\CC_X)$ of $\SM$ 
has a normal form at $a$.
\end{proposition}

Let $\SM$ be a holonomic $\SD_X$-module and $N$ the 
multiplicity for which $Sol_X^{\rmE}(\SM)$ has a normal form at $a\in X$.
Then we call the sections of 
$N^{>0}:=N^{-1}( \ZZ_{>0} )\subset\SP_{S_aX}^\prime$ 
the exponential factors of $\SM$ at $a$. 

\begin{proposition}[{\cite[Proposition 5.4.5]{DK18}}]\label{prop:hariaw}
Let $\SM\in\Modhol(\SD_X)$ and $a\in X$.
If $Sol_X^{\rmE}(\SM)\in\ECrc^{\rmb}(\rmI\CC_X)$ has a normal 
form at $a\in X$ for a multiplicity $N\colon
\SP_{S_aX}^\prime\to(\ZZ_{\geq0})_{S_aX}$, then there exist an open neighborhood 
$\Omega$ of $a$ and
an enhanced sheaf $F(a)\in\ECrc^0(\CC_X)$ such that 
$F(a)$ has a normal form at $a$ for the multiplicity $N$ and there exists an isomorphism
\begin{align}
\pi^{-1}\CC_{\Omega\bs\{a\}}\otimes Sol_X^{\rmE}(\SM) 
\simeq \CC_X^{\rmE}\Potimes F(a).
\end{align}
\end{proposition}
In \cite[Section 3]{KT23} we refined Proposition \ref{prop:hariaw} 
to get also a global explicit
description of the enhanced solution complex of a localized holonomic 
$\SD$-module i.e. a meromorphic connection 
on a complex manifold of dimension one.

\subsection{Theory of monodromy zeta functions}\label{sec:zeta}

In this subsection, we recall the 
theory of monodromy zeta functions, which will be 
effectively used to prove our main theorems. In what follows, let 
$X$ be a complex manifold and $f\colon X \longrightarrow \CC$ 
a non-constant holomorphic function on it.  Set 
$X_0:=f^{-1}(0)\subsetneqq X$ and let 
\begin{equation}
\psi_f( \cdot ): \Db(\CC_X) \longrightarrow \Db(\CC_{X_0})
\end{equation}
be Deligne's nearby cycle functor associated to $f$ 
(see e.g. \cite[Section 4.2]{Dim04} and 
\cite[Section 2]{Tak23} for 
the details). 
As it preserves the construtibility, we obtain also 
a functor 
\begin{equation}
\psi_f( \cdot ): \Dbc(\CC_X) \longrightarrow \Dbc(\CC_{X_0}). 
\end{equation}
For $\F \in \Db (\CC_X)$ there exists also a monodromy automorphism 
\begin{equation}
\Phi ( \F ) \colon \psi_f( \F ) \simto \psi_f( \F )
\end{equation} 
of $\psi_f( \F )$ in $\Db (\CC_{X_0})$ 
(See e.g. \cite[Section 4.2]{Dim04} and 
\cite[Section 2]{Tak23} for the definition. 
See also Remark \ref{remmon} below). 
Recall that by Milnor's fibration theorem for $0<\eta\ll\varepsilon\ll 1$ the 
restriction 
\begin{equation}
f^{\circ}: B(x;\varepsilon)\cap f^{-1}(D_{\eta}^*) \longrightarrow D_{\eta}^* 
\quad (\subset\CC^*) 
\end{equation}
of $f$ is a fiber bundle, where $ B(x;\varepsilon) \subset X$ is 
an open ball in $X$ with radius $\varepsilon >0$ centered at $x \in X_0$ and 
we set $D_{\eta}^*:= 
\{t\in\CC \mid 0<|t|<\eta \} \subset \CC^*= \CC \setminus \{ 0 \}$. We call a 
fiber of this fiber bundle the Milnor fiber of 
$f$ at $x \in X_0$ and denote it by $F_x \subset X \setminus X_0$. 
In this situation, for any $j \in \ZZ$ the direct image 
sheaf $H^j {\rm R}f^{\circ}_* \CC_X$ is a local system on 
the punctured disk $D_{\eta}^* 
\subset\CC^*$ and there exist isomorphisms 
\begin{equation}
(H^j {\rm R}f^{\circ}_* \CC_X)_t \simeq H^j((f^{\circ})^{-1}(t) 
; \CC) \qquad 
(t \in D_{\eta}^*). 
\end{equation}
Then by the parallel translations of the sections of the 
local system $H^j {\rm R}f^{\circ}_* \CC_X$ along a 
loop $C$ in $D_{\eta}^*$ in the ``counterclockwise" 
direction, for a point $t \in C$ and $F_x=(f^{\circ})^{-1}(t)$ 
we obtain an automorphism of the cohomology group 
$H^j(F_x; \CC )$. We call it the (cohomological) 
Milnor monodromy (in degree $j \in \ZZ$). It is 
uniquely determined up to conjugacy. 
By the following basic results we can study 
Milnor fibers and their monodromies via nearby cycle functors. 

\begin{theorem}\label{thm.7.2.004}
{\rm (see e.g. Dimca \cite[Proposition 4.2.2]{Dim04} and 
\cite[Theorem 2.6]{Tak23})} 
For a point 
$x \in X_0=f^{-1}(0)\subset X$ of 
$X_0$ let $F_x \subset X \setminus X_0$ be the 
Milnor fiber of $f$ at $x$. Then there exist isomorphisms  
\begin{equation}\label{Mil-isom} 
H^j(F_x; \CC ) \simeq H^j \psi_f( \CC_X )_x \qquad (j\in\ZZ). 
\end{equation}
Moreover these isomorphisms are compatible 
with the automorphisms of both sides 
induced by the monodromies. 
\end{theorem}

\begin{remark}\label{remmon} 
We choose the monodromy automorphism $\Phi ( \CC_X )$ of 
the nearby cycle sheaf 
$\psi_f( \CC_X )$ so that the isomorphisms in 
\eqref{Mil-isom} are compatible with the 
Milnor monodromies on $H^j(F_x; \CC )$. 
\end{remark}
We have also the following more general result. 

\begin{theorem}\label{thm.7.2.006}
{\rm (see e.g. Dimca \cite[Proposition 4.2.2]{Dim04} and 
\cite[Theorem 2.7]{Tak23})} 
In the situation of Theorem \ref{thm.7.2.004}, 
for any constructible sheaf $\F \in \Dbc (\CC_X)$ on $X$ 
there exist isomorphisms 
\begin{equation}
H^j(F_x; \F ) \simeq H^j \psi_f( \F )_x 
 \qquad (j\in\ZZ).
\end{equation}
\end{theorem}

As was first observed by A'Campo \cite{A'Campo}, 
Milnor fibers and their monodromies can be precisely studied 
by resolutions of singularities of 
$X_0=f^{-1}(0) \subset X$. Here we use his idea 
in the more general framework 
of constructible sheaves. 
Let $\CC(t)^*=\CC(t) \setminus \{0\}$ be the 
multiplicative group of the function field $\CC(t)$ of 
one variable $t$. 

\begin{definition}\label{dfn:2-79}
For a point $x \in X_0=f^{-1}(0)$ of $X_0$ we define 
the monodromy zeta function $\zeta_{f,x}(t) \in \CC(t)^*$ 
of $f$ at $x$ by 
\begin{equation}
\zeta_{f,x}(t):=\prod_{j=0}^{\infty} \ 
 \Bigl\{ \det(\id -t \cdot \Phi_{j,x}) \Bigr\}^{(-1)^j} 
\quad \in \CC(t)^*, 
\end{equation}
where the $\CC$-linear maps 
$\Phi_{j,x} \colon H^j(F_x; \CC ) \simto H^j(F_x; \CC )$ are 
the Milnor monodromies at $x$. 
\end{definition}

This monodromy zeta function $\zeta_{f,x}(t)$ is 
related to the characteristic polynomials of the 
Milnor monodromies $\Phi_{j,x} \colon H^j(F_x; \CC ) \simto H^j(F_x; \CC )$ 
($j=0,1,2, \ldots$) as follows. 
For a point $x \in X_0=f^{-1}(0)$ of $X_0$ we set 
\begin{equation}
\widetilde{\zeta}_{f,x}(t):=\prod_{j=0}^{\infty} \ 
 \Bigl\{ \det( t \cdot \id - \Phi_{j,x}) \Bigr\}^{(-1)^j} 
\quad \in \CC(t)^*. 
\end{equation}
Then for the Euler characteristic $\chi (F_x)
= {\rm deg} \zeta_{f,x}$ of the Milnor fiber 
$F_x$ of $f$ at $x$ we can easily show the following equalities 
\begin{equation}
\widetilde{\zeta}_{f,x}(t) = t^{\chi (F_x)} \cdot 
\zeta_{f,x} \Bigl( \frac{1}{t} \Bigr), 
\qquad 
\zeta_{f,x}(t) = t^{\chi (F_x)} \cdot 
\widetilde{\zeta}_{f,x} \Bigl( \frac{1}{t} \Bigr). 
\end{equation}
The following lemma was obtained by A'Campo \cite{A'Campo} 
(see also Oka \cite[Chapter I, Example (3.7)]{Oka97} for a 
precise explanation). 

\begin{lemma}\label{lem:2-ac-21}
For $1 \leq k \leq n$ let $h: \CC^n \longrightarrow \CC$ be  
the function on $\CC^n$ defined by $h(z)=z_1^{m_1}z_2^{m_2} 
\cdots z_k^{m_k}$ 
($m_i \in \ZZ_{>0}$) for $z=(z_1,z_2, \ldots, z_n) \in \CC^n$. 
Then we have 
\begin{equation}
\zeta_{h,0}(t)=
\begin{cases}
\ 1-t^{m_1} & (k=1) \\
& \\
\ 1 & (k>1) \\
\end{cases}
\end{equation} 
\end{lemma}
The classical notion of monodromy zeta functions can be generalized as follows.

\begin{definition}\label{dfn:2-8}
For a constructible sheaf $\F \in \Dbc(\CC_X)$ on $X$ and 
a point $x \in X_0=f^{-1}(0)$ of $X_0$ we define 
the monodromy zeta function $\zeta_{f,x}(\F) \in \CC(t)^*$ of 
$\F$ along $f$ at $x$ by 
\begin{equation}
\zeta_{f,x}(\F)(t):=\prod_{j \in \ZZ} \ 
\Bigl\{ \det\left(\id -t \cdot \Phi(\F)_{j,x}\right) \Bigr\}^{(-1)^j} 
\quad \in \CC(t)^*, 
\end{equation}
where the $\CC$-linear maps 
$\Phi(\F)_{j,x} \colon H^j(\psi_f(\F))_x \simto H^j(\psi_f(\F))_x$ are induced by 
$\Phi(\F)$.
\end{definition}
In the situation of Definition \ref{dfn:2-8}, for a 
point $x \in X_0=f^{-1}(0)$ of $X_0$ we set 
\begin{equation}
\widetilde{\zeta}_{f,x}(\F)(t):=\prod_{j \in \ZZ} \ 
\Bigl\{ \det\left( t \cdot \id - \Phi(\F)_{j,x}\right) \Bigr\}^{(-1)^j} 
\quad \in \CC(t)^*
\end{equation}
and 
\begin{equation}
\chi_x( \psi_f(\F) )
:=\sum_{j \in \ZZ} \ (-1)^j {\rm dim} H^j \psi_f(\F)_x 
= {\rm deg} \zeta_{f,x} ( \F ) \qquad \in \ZZ. 
\end{equation}
Then similarly we obtain the following equality 
\begin{equation}\label{invrel} 
\widetilde{\zeta}_{f,x}(\F)(t) = t^{\chi_x( \psi_f(\F) )} \cdot 
\zeta_{f,x}(\F)  \Bigl( \frac{1}{t} \Bigr). 
\end{equation}
If for some $p \in \ZZ$ we have $H^j \psi_f(\F)_x \simeq 0$ 
($j \not= p$) and $H^p \psi_f(\F)_x \not= 0$, then by 
\eqref{invrel} we obtain the characteristic polynomial 
\begin{equation}
\det\left( t \cdot \id - \Phi(\F)_{p,x}\right) \qquad 
\in \ZZ
\end{equation}
of the only non-trivial monodromy operator 
\begin{equation}
\Phi(\F)_{p,x} \colon H^p(\psi_f(\F))_x \simto H^p(\psi_f(\F))_x
\end{equation}
by multiplying some powers of $t$ to 
$\Bigl\{ \zeta_{f,x}(\F) \Bigl( \frac{1}{t} \Bigr) \Bigr\}^{(-1)^p}$. 
As in A'Campo \cite{A'Campo} (see also Oka \cite[Chapter I]{Oka97}) 
we can easily prove the following very useful result. 

\begin{lemma}\label{lemma:2-8}
Let $\F^{\prime} \longrightarrow \F \longrightarrow \F^{\prime \prime}
 \overset{+1}{\longrightarrow}$ be a distinguished triangle 
in $\Dbc (\CC_X)$. Then for any point $x \in X_0=f^{-1}(0)$ of $X_0$ we have 
\begin{equation}
\zeta_{f,x}( \F )(t)= \zeta_{f,x}( \F^{\prime} )(t) 
\cdot \zeta_{f,x}( \F^{\prime \prime} )(t). 
\end{equation}
\end{lemma} 

\begin{remark}\label{dist-lem}
Lemma \ref{lemma:2-8} holds true also for 
the rational functions $\widetilde{\zeta}_{f,x}( \cdot ) 
\in \CC (t)^*$. 
\end{remark}

For an abelian group $G$ and a complex analytic space $Z$,  
we shall say that a $G$-valued function 
$\varphi :Z \longrightarrow G$ on $Z$ is constructible if 
there exists a stratification $Z=\bigsqcup_{\alpha}Z_{\alpha}$ of $Z$ 
such that $\varphi |_{Z_{\alpha}}$ is 
constant for any $\alpha$. We denote by $\CF_G(Z)$ 
the abelian group of $G$-valued constructible functions on $Z$. 
In this paper, we consider $\CF_G(Z)$ only for 
the multiplicative group $G= \CC(t)^*$. 
Then for the complex manifold $X$ and the non-constant 
holomorphic function $f :X \longrightarrow \CC$ on it, 
by Theorem \ref{thm.7.2.006}  
we can easily see that the $\CC(t)^*$-valued function 
$\zeta_{f}(\F) :X_0 \longrightarrow \CC(t)^*$ on $X_0=f^{-1}(0)$ 
defined by 
\begin{equation}
\zeta_{f}(\F)(x):=\zeta_{f,x}(\F)(t) \quad \in \CC(t)^* \qquad 
(x \in X_0=f^{-1}(0))
\end{equation}
is constructible. 
For a $G$-valued constructible function $\varphi \colon Z \longrightarrow G$ on 
a complex analytic space $Z$, by taking 
a stratification $Z=\bigsqcup_{\alpha}Z_{\alpha}$ of 
$Z$ such that $\varphi |_{Z_{\alpha}}$ is 
constant for any $\alpha$, we set 
\begin{equation}
\int_Z \varphi :=\sum_{\alpha}\chi(Z_{\alpha}) \cdot 
\varphi (z_{\alpha}) \quad \in G, 
\end{equation}
where $\chi ( \cdot )$ stands for the 
topological Euler characteristic and $z_{\alpha}$ is a reference point in 
the stratum $Z_{\alpha}$. 
By the following lemma $\int_Z \varphi \in G$ does not depend on the choice of 
the stratification $Z=\bigsqcup_{\alpha} Z_{\alpha}$ of $Z$. 
We call it the topological (or Euler) integral of 
$\varphi$ over $Z$. 

\begin{lemma}\label{lem:2-str-1}
Let $Y$ be a complex analytic space and 
$Y=\bigsqcup_{\alpha}Y_{\alpha}$ a stratification of $Y$. 
Then we have 
\begin{equation}
\chi_{{\rm c}}(Y) = \dsum_{\alpha} \chi_{{\rm c}}(Y_{\alpha}), 
\end{equation}
where $\chi_{{\rm c}}( \cdot )$ stands for the Euler 
characteristic with compact supports. Moreover, 
for any $\alpha$ we have 
$\chi_{{\rm c}}(Y_{\alpha})=\chi (Y_{\alpha})$. 
\end{lemma}

More generally, 
for any morphism $\rho \colon Z \longrightarrow W$ of complex analytic spaces 
and any $G$-valued constructible function 
$\varphi \in \CF_G(Z)$ on $Z$, we define the 
push-forward $\int_{\rho} \varphi \in \CF_G(W)$ of 
$\varphi$ by 
\begin{equation}
\Bigl( \int_{\rho} \varphi \Bigr) (w) :=\int_{\rho^{-1}(w)} \varphi \qquad 
(w \in W). 
\end{equation}
We thus obtain a homomorphism 
\begin{equation}
\int_{\rho} : \CF_G(Z) \longrightarrow \CF_G(W)
\end{equation}
of abelian groups. Then we have the 
following very useful result.   

\begin{theorem}\label{prp:2-99}
{\rm (Dimca \cite[Proposition 4.2.11]{Dim04} 
and Sch\"urmann \cite[Chapter 2]{Sch03})}  
Let $\rho \colon Y  \longrightarrow X$ be a 
proper morphism of complex manifolds  
and $f\colon X  \longrightarrow \CC$ a non-constant 
holomorphic function on $X$. We set 
$g:=f\circ\rho\colon Y  \longrightarrow \CC$ and 
\begin{equation}
X_0=f^{-1}(0)\subset X,\quad Y_0=g^{-1}(0)=\rho^{-1}(X_0)\subset Y. 
\end{equation}
Let $\rho|_{Y_0}\colon Y_0 \longrightarrow X_0$ be the 
restriction of $\rho$ to $Y_0\subset Y$ and 
assume that $\G\in \Db (\CC_Y)$ is constructible. 
Then we have 
\begin{equation}
\int_{ \rho|_{Y_0}} \zeta_g(\G) =\zeta_f({\rm R} \rho_*\G)
\end{equation}
in $\CF_{\CC(t)^*}(X_0)$, where $\int_{\rho|_{Y_0}}\colon \CF_{\CC(t)^*}(Y_0) 
\longrightarrow \CF_{\CC(t)^*}(X_0)$ is the push-forward of 
$\CC(t)^*$-valued constructible functions by 
$\rho|_{Y_0} \colon Y_0 \longrightarrow X_0$.
\end{theorem}
For the proof of this theorem, 
see for example, \cite[p.170-173]{Dim04} and 
\cite[Chapter 2]{Sch03}. 

\begin{corollary}\label{cpr:2-99}
In the situation of Theorem \ref{prp:2-99}, 
for a point $x \in X_0=f^{-1}(0)$ of $X_0$ let 
$\rho^{-1}(x)= \sqcup_{\alpha} Z_{\alpha}$ be 
a stratification of $Z:= \rho^{-1}(x) \subset 
Y_0=g^{-1}(0)$ such that for any $\alpha$ 
the $\CC(t)^*$-valued constructible function 
$\zeta_{g}( \G ) \in \CF_{\CC(t)^*}(Y_0)$ is 
constant on $Z_{\alpha}$. Then we have 
\begin{equation}
\zeta_{f,x}( {\rm R} \rho_* \G )(t)= 
\dprod_{\alpha} \Bigl\{
\zeta_{g, y_{\alpha}}( \G )(t) 
\Bigr\}^{\chi ( Z_{\alpha} )}, 
\end{equation}
where $y_{\alpha}$ is a reference point of $Z_{\alpha}$. 
\end{corollary}

The following two lemmas are useful to calculate monodromy 
zeta functions of constructible sheaves by 
Corollary \ref{cpr:2-99}. Indeed, 
by resolutions of singularities, we can reduce the 
problem to the situations treated in them. 

\begin{lemma}\label{lem:2-99}
{\rm (\cite[Proposition 5.2]{MT11a})} 
For $1 \leq k \leq n$ let $L$ be a $\CC$-local system 
of rank $r>0$ on $( \CC^*)^k \times \CC^{n-k} \subset \CC^n$. 
Let $j: ( \CC^*)^k \times \CC^{n-k} \hookrightarrow \CC^n$
be the inclusion map and $h: \CC^n \longrightarrow \CC$ 
the function on $\CC^n$ defined by $h(z)=z_1^m$ 
($m \in \ZZ_{>0}$) for $z=(z_1,z_2, \ldots, z_n) \in \CC^n$. 
\begin{enumerate}
\item[\rm{(i)}] Assume moreover that $k>1$. Then we have 
$\psi_h(j_! L)_0 \simeq 0$ and hence $\zeta_{h,0}(j_! L )(t)=1$.  
\item[\rm{(ii)}]  Assume moreover that $k=1$ and 
let $A \in {\rm GL}_r( \CC )$ be the monodromy matrix of 
$L$ along the loop $C:= \{ (e^{\sqrt{-1} \theta}, 0, \ldots, 0) 
\ | \ 0 \leq \theta \leq 2 \pi \}$ in 
$( \CC^*)^k \times \CC^{n-k}= 
\CC^* \times \CC^{n-1}$ (defined up to conjugacy). 
Then we have 
\begin{equation}
\zeta_{h,0}(j_! L )(t)= 
{\rm det} \bigl( {\rm id}- t^m A \bigr) \quad \in \CC (t)^*. 
\end{equation}
\end{enumerate}
\end{lemma}
The following result, which is a 
generalization of A'Campo's lemma i.e. 
Lemma \ref{lem:2-ac-21} 
(see also Oka \cite[Example (3.7)]{Oka97}) 
to constructible sheaves. 

\begin{lemma}\label{lem:2-999}
{\rm (\cite[Proposition 5.3]{MT11a})} 
For $0 \leq k \leq n$ let $L$ be a $\CC$-local system 
of rank $r>0$ on $( \CC^*)^k \times \CC^{n-k} \subset \CC^n$. 
Let $j: ( \CC^*)^k \times \CC^{n-k} \hookrightarrow \CC^n$
be the inclusion map and $h: \CC^n \longrightarrow \CC$ 
the function on $\CC^n$ defined by $h(z)=z_1^{m_1}z_2^{m_2} 
\cdot z_n^{m_n}$ 
($m_i \in \ZZ_{\geq 0}$) for $z=(z_1,z_2, \ldots, z_n) \in \CC^n$. 
Assume that $\sharp \{ 1 \leq i \leq n \ | \ m_i>0 \} 
\geq 2$. Then we have $\zeta_{h,0}(j_! L )(t)=1$.  
\end{lemma}

For our use in the proof of Theorem \ref{MT-3}, we recall the following 
proposition for the Euler characteristic of a hypersurface in a projective space.

\begin{proposition}{\rm (\cite[Section 5.3]{Dim92})}\label{prp:2-999}
Let $f(z)\in\CC[z_1,z_2,\dots,z_{N+1}]$ be 
a homogeneous polynomial of degree $d>0$ such that the hypersurface
$D \coloneq\{f=0\}\subset \PP^N$ is smooth.
Then we have
\begin{equation}
\chi(D)=\frac{1}{d}\{(1-d)^{N+1}-1\}+N+1.
\end{equation} 
\end{proposition}

\section{Proof of Theorem \ref{MT-1}}\label{dim-1}
In this section, we prove Theorem \ref{MT-1} and 
give some examples. 
Set $D\coloneq\ssupp(\SM)\subset X$ and 
$U\coloneq X\bs D\subset X$.
Then there exists a distinguished triangle 
\begin{equation}
\rsect_D(\SM)\longrightarrow\SM
\longrightarrow\rsect_{X\bs D}(\SM)
\overset{+1}{\longrightarrow}.
\end{equation}
The inclusion map $U\hookrightarrow X$ being affine, we have 
\begin{equation}
H^j\rsect_{X\bs D}(\SM)\simeq0 
\quad (j\neq0)
\end{equation}
and $H^0\rsect_{X\bs D}(\SM)\simeq\Gamma_{X\bs D}(\SM)$ is 
the localization of $\SM$ along the divisor $D\subset X$.
We thus obtain an exact sequence
\begin{equation}
0\longrightarrow\Gamma_D(\SM)\longrightarrow
\SM\longrightarrow\Gamma_{X\bs D}(\SM)
\longrightarrow H_D^1(\SM)\longrightarrow0.
\end{equation} 
Since the Fourier transform $(\cdot )^\wedge$ is 
an exact functor and for any $j\in\{0,1\}$ 
the Fourier transform $(H_D^j(\SM))^\wedge$ of 
$H_D^j(\SM)\in\Modhol(\SD_X)$ has 
the trivial monodromy at infinity, 
we may assume that $\SM\simeq\Gamma_{X\bs D}(\SM)$ 
i.e. $\SM$ is an algebraic meromorphic connection.
Indeed, let $i\colon Y^\an\hookrightarrow\var{Y}^\an$ be 
the inclusion map and $h$ a holomorphic coordinate of 
$\var{Y}^\an\simeq\PP^1$ on a 
neighborhood of the point $\infty\in\var{Y}^\an$ 
such that $h^{-1}(0)=\{\infty\}$. 
Then for the study of the monodromy at 
infinity of $\SM^\wedge$ it suffices to calculate the monodromy zeta functions 
\begin{equation}
\zeta_{h,\infty}(i_! Sol_Y(\SM^\wedge )), \ 
\tl{\zeta}_{h,\infty}(i_! Sol_Y(\SM^\wedge ))
\quad \in \CF_{\CC(t)^\ast}(\{\infty\})\simeq\CC(t)^\ast
\end{equation}
and hence by Lemma \ref{lemma:2-8} and Remark \ref{dist-lem} 
we can show that for any $\lambda\in\CC^\ast$ such that 
$\lambda\neq1$ we have
\begin{equation}
\mu (\SM^\wedge ,\lambda)= \mu (\Gamma_{X\bs D}(\SM)^\wedge 
,\lambda).
\end{equation}
In what follows, we thus assume that $\SM$ is an 
algebraic meromorphic connection. 
Denote by $L$ the local system 
$Sol_X(\SM)\vbar_{U^\an}\simeq H^0Sol_X(\SM)\vbar_{U^\an}$ on 
$U^\an$. Recall that by the inclusion map 
$i_X\colon X\hookrightarrow\var{X}$ 
we set $\tl{\SM}\coloneq\bfD i_{X\ast}\SM\simeq 
i_{X\ast}\SM\in\Modhol(\SD_{\var{X}})$. 
Let $a_1,a_2,\dots,a_l\in D^{\an}$
be the points in $D^{\an}\subset X^{\an}$
and set $a_\infty\coloneq\infty\in\var{X}^{\an}$
so what we have 
\begin{equation}
\tl{D}\coloneq  \mathrm{sing.supp}(\tl{\SM})^{\an} = 
D^{\an}  \sqcup\{a_\infty\} 
=\{a_1,a_2,\dots,a_l,a_\infty\}.
\end{equation}
For a point $a_i\in \tl{D}$ let $B(a_i)^\circ 
\subset \var{X}^{\an}$ be 
a sufficiently small open punctured disk centered 
at it and 
\begin{align}
N_i =N(a_i) \colon \SP_{S_{a_i}\var{X}^{\an}}^\prime
\longrightarrow (\ZZ_{\geq 0})_{S_{a_i}\var{X}^{\an}}
\end{align}
the multiplicity of 
the analytic meromorphic connection
$(\tl{\SM})^{\an}$ at it. 
If an exponential factor $f \in N_i^{>0}$ at the point $a_i \in 
\tl{D}$ is holomorphic on an open subset $W \subset B(a_i)^\circ$, 
by abuse of notations we write $f \in N_i^{>0}(W)$. 
For such $f\in N_i^{>0}(W)$ and $w \in Y^{\an}= \CC$ 
we define a holomorphic function $f^w$ on $W$ by
\begin{align}
f^w(z) \coloneq zw-f(z) \quad \(z\in W \).
\end{align}
Recall that in \cite[Section 3]{KT23} we constructed 
an enhanced sheaf $G\in\BECrc(\CC_{\var{X}^\an})$ 
on $\var{X}^\an\simeq\PP^1$ 
explicitly such that 
\begin{equation}
\SolE_{\var{X}}(\tl{\SM})\simeq \CCE_{\var{X}^\an}\Potimes G.
\end{equation}
Note that by our construction of 
$G$ the restriction of 
$\rmR \pi_\ast G$ to $U^\an\subset \var{X}^\an$ is isomorphic to 
the local system $L$. 
Set $G^\circ\coloneq G\vbar_{X^\an\times\RR_s}$ and let 
\begin{equation}
X^\an\times\RR_s\overset{p_1}{\longleftarrow}
(X^\an\times\RR_s)\times(Y^\an\times\RR_t)
\overset{p_2}{\longrightarrow}Y^\an\times\RR_t
\end{equation}
be the projections.
Then for 
the (enhanced) Fourier-Sato transform 
${}^\sfL G\in\BECrc(\CC_{\var{Y}^\an})$ of $G$ 
(for the definition, see \cite[Section 4.1]{KT23}) 
on $\var{Y}^\an$ 
we have an isomorphism 
\begin{equation}
{}^\sfL G\vbar_{Y^\an\times\RR_t}\simeq
\bfQ(\rmR p_{2!}(p_1^{-1}G^\circ\otimes\CC_{\{t-s-\Re zw\geq0\}}[1])),
\end{equation}
where $\bfQ\colon\BDC(\CC_{Y^\an\times\RR_t})
\longrightarrow\BEC(\CC_{Y^\an})$ 
is the quotient functor. Recall that we 
have an isomorphism 
\begin{equation}
Sol_{Y}^\rmE( \SM^{\wedge}) \simeq \CCE_{Y^\an}\Potimes 
({}^\sfL G\vbar_{Y^\an\times\RR_t}) 
\end{equation}
and hence for our study of the Fourier transform $\SM^{\wedge}$ 
it suffices to study ${}^\sfL G\vbar_{Y^\an\times\RR_t}$. 
Moreover, for a point 
$(w,t)\in Y^{\an}\times\RR$
we have an isomorphism
\begin{align}
\(\rmR p_{2!}(p_1^{-1}G^{\circ} \otimes
\CC_{\{t-s-\Re zw\geq0\}}[1])\)_{(w,t)}
\simeq \rsect_c \(X^{\an};
\rmR\pi_!(G^{\circ} \otimes\CC_{\{t-s-\Re zw\geq0\}})[1]\).
\end{align}
Let us fix a point $(w,t)\in Y^{\an} \times\RR$ and explain the structure of
\begin{align}
G(w,t) \coloneq
\rmR\pi_!(G^{\circ} \otimes\CC_{\{t-s-\Re zw\geq0\}} [1])\
\in\BDC(\CC_{X^{\an}}).
\end{align}
By our construction of $G$,
the restriction of $G(w,t)$ to
$D^{\an}=\{a_1,\dots,a_l\}\subset X^{\an}$ is zero and  
on $U^{\an}=X^{\an} \setminus D^{\an}$ we have 
\begin{equation}\label{vanshi} 
H^j \rmR\pi_!(G^{\circ} \otimes\CC_{\{t-s-\Re zw\geq0\}}) \simeq 0 
\qquad (j \not= 0). 
\end{equation}
We thus obtain $H^j G(w,t) \simeq 0$ for $j \not= -1$ and 
on $U^{\an}$ there 
exists an isomorphism 
\begin{align}\label{tru-vic} 
\rmR\pi_!(G^{\circ} \otimes\CC_{\{t-s-\Re zw\geq0\}}) 
\simto 
\rmR\pi_*(G^{\circ} \otimes\CC_{\{t-s-\Re zw\geq0\}}). 
\end{align}
This implies that 
the canonical morphism $G^{\circ} \longrightarrow 
G^{\circ} \otimes\CC_{\{t-s-\Re zw\geq0\}}$ induces 
a morphism 
\begin{align}\label{surmorp} 
\Phi (w,t): 
L \simeq H^0( \rmR\pi_* G^{\circ})|_{U^{\an}}  \longrightarrow 
M(w,t):= 
H^{-1}G(w,t)|_{U^{\an}} 
\end{align}
of $\RR$-constructible sheaves on $U^{\an}$. In \cite[Section 4.2]{KT23}
we showed that it is surjective. Namely $M(w,t)$ 
is a quotient sheaf of the local system $L$. 
More precisely, for any simply connected 
open subset $W \subset B(a_i)^\circ$ 
there exists an isomorphism 
\begin{align}\label{locdesc} 
M(w,t)|_W \simeq 
\bigoplus_{f \in N_i^{>0}(W)} 
\Bigl(\CC_{\left\{z\in W  \mid 
\Re f^w(z) \leq t\right\}} \Bigr)^{N_i(f)}
\end{align}
(see \cite[Section 4.2]{KT23}). 
Then for the kernel $K(w,t):= {\rm Ker} \Phi (w,t) 
\subset L$ of $\Phi (w,t)$ and any simply connected 
open subset $W \subset B(a_i)^\circ$ 
we obtain an isomorphism 
\begin{align}\label{locdescofK} 
K(w,t)|_W \simeq 
\bigoplus_{f \in N_i^{>0}(W)} 
\Bigl(\CC_{\left\{z\in W  \mid 
\Re f^w(z) > t\right\}} \Bigr)^{N_i(f)}. 
\end{align}
In our proof of Theorem \ref{MT-1}, this local description of 
the sheaf $K(w,t) \subset L$ will play a central role. Let 
\begin{equation}
\tl{\pi}\colon(U^\an\times\RR_s)\times(Y^\an\times\RR_t)\longrightarrow
U^\an\times(Y^\an\times\RR_t)
\end{equation}
be the projection and 
\begin{equation}
j\colon(U^\an\times\RR_s)\times(Y^\an\times\RR_t)
\longhookrightarrow (X^\an\times\RR_s)\times(Y^\an\times\RR_t)
\end{equation}
the inclusion map.
Then by the construction of $G^\circ$ we have 
\begin{equation}
j_! j^{-1}(p_1^{-1}G^\circ)\simto p_1^{-1}G^\circ
\end{equation}
and hence 
\begin{equation}\label{eq:cutU}
j_! j^{-1}(p_1^{-1}G^\circ\otimes\CC_{\{t-s-\Re zw\geq0\}})
\simto p_1^{-1}G^\circ\otimes\CC_{\{t-s-\Re zw\geq0\}}.
\end{equation}
As in \eqref{vanshi} we have  
\begin{equation}
H^j\rmR \tl{\pi}_! j^{-1}(p_1^{-1}G^\circ
\otimes\CC_{\{t-s-\Re zw\geq0\}}) \simeq0. 
\quad (j\neq0)
\end{equation}
Moreover, as in \eqref{tru-vic} there exists an isomorphism
\begin{equation}
\rmR \tl{\pi}_! j^{-1}(p_1^{-1}G^\circ
\otimes\CC_{\{t-s-\Re zw\geq0\}})
\simto
\rmR \tl{\pi}_\ast j^{-1}(p_1^{-1}G^\circ
\otimes\CC_{\{t-s-\Re zw\geq0\}}).
\end{equation}
Since the restriction of 
$\rmR \pi_\ast G^\circ$ to $U^\an\subset X^\an$ is isomorphic to $L$, 
for the projection 
$ q_1\colon U^\an\times(Y^\an\times\RR_t)\longrightarrow U^\an$ 
we have an isomorphism
\begin{equation}
\rmR \tl{\pi}_\ast j^{-1}(p_1^{-1}G^\circ)\simeq
q_1^{-1} L.
\end{equation} 
This implies that the canonical morphism 
$p_1^{-1}G^\circ\longrightarrow p_1^{-1}G^\circ\otimes\CC_{\{t-s-\Re zw\geq0\}}$ 
induces a morphism
\begin{equation}\label{eq:ast!}
\Phi: q_1^{-1} L \simeq 
\rmR \tl{\pi}_\ast j^{-1}(p_1^{-1}G^\circ) \longrightarrow
M:= H^0 \rmR \tl{\pi}_! j^{-1}(p_1^{-1}G^\circ
\otimes\CC_{\{t-s-\Re zw\geq0\}})
\end{equation}
of $\RR$-constructible sheaves on $U^\an\times(Y^\an\times\RR_t)$. 
We see also that the morphism $\Phi$ is surjective and 
hence $M$ is a quotient sheaf of the local system $q_1^{-1} L$. 
Set $K\coloneq\Ker\Phi$ and $\tl{L}\coloneq q_1^{-1}L$ and 
consider the exact sequence
\begin{equation}
0\longrightarrow K \longrightarrow \tl{L}
\overset{\Phi}{\longrightarrow} M \longrightarrow 0.
\end{equation}
Then for any point $(w,t) \in Y^\an\times\RR$ the restriction 
of $M$ (resp. $K$) to the submanifold  
$U^{\an} \times \{ (w,t) \} \simeq U^{\an}$ of 
$U^\an\times(Y^\an\times\RR_t)$ is isomorphic to 
$M(w,t)$ (resp. $K(w,t)$). 
Let $q_2\colon U^\an\times(Y^\an\times\RR_t)
\longrightarrow(Y^\an\times\RR_t)$ be the projection.
Then by (\ref{eq:cutU}) we obtain an isomorphism
\begin{equation}
{}^\sfL G \vbar_{Y^\an\times\RR_t} =
\rmR p_{2!}(p_1^{-1}G^\circ\otimes\CC_{\{t-s-\Re zw\geq0\}}) [1] 
\simeq \rmR q_{2!} M [1].
\end{equation}
We thus obtain a distinguished triangle 
\begin{equation}
\rmR q_{2!} K\longrightarrow
\rmR q_{2!} \tl{L}\longrightarrow
{}^\sfL G \vbar_{Y^\an\times\RR_t} [-1] 
\overset{+1}{\longrightarrow}.
\end{equation}
Let $r \geq0$ be the generic rank 
$\rank\SM^\wedge$ of $\SM^\wedge$. 
For the point 
$b_\infty\coloneq\infty \in\var{Y}^\an\simeq\PP^1$ 
let $B(b_\infty)^\circ\subset\var{Y}^\an\bs\{\infty\}=Y^\an$ 
be a sufficiently small open punctured disk 
in $\var{Y}^\an\simeq\PP^1$ centered at it.
Then by \cite[Proposition 4.10]{KT23} there exists 
a continuous function $\phi\colon B(b_\infty)^\circ\longrightarrow\RR$ 
such that the restriction of ${}^\sfL G \vbar_{Y^\an\times\RR_t}
 \simeq\rmR q_{2!} M[1]$ 
to the subset 
\begin{equation}
E\coloneq\{(w,t)\mid w\in B(b_\infty)^\circ, t\geq\phi(w)\}
\quad \subset Y^\an\times\RR
\end{equation}
of $Y^\an\times\RR_t$ is concentrated in degree zero 
and a local system of rank $r$.
The restrictions of the cohomology sheaves of 
$\rmR q_{2!}\tl{L}$ to $E$ are also 
locally constant (actually they are constant).
Hence the same is true also for $\rmR q_{2!} K$.
Moreover by the proof of \cite[Corollary 4.5]{IT20a} 
and \cite[Theorem 1.1]{KT23}
on the punctured disk $B(b_\infty)^\circ\subset Y^\an$ 
there exists an isomorphism
\begin{equation}
\rmR \pi_\ast({}^\sfL G \vbar_{Y^\an\times\RR_t})_E \simeq
Sol_Y(\SM^\wedge).
\end{equation}
As before, let $i\colon Y^\an\longhookrightarrow\var{Y}^\an$ be 
the inclusion map and $h$ a holomorphic function on the open disk 
$B(b_\infty)\coloneq B(b_\infty)^\circ\sqcup\{\infty\}\subset\var{Y}^\an$ 
centered at the point $b_\infty=\infty\in\var{Y}^\an$ such that 
$(h)^{-1}(0)=\{b_\infty\}$.
Then applying the functor 
$\psi_{h,\infty}(i_!(\cdot))$ to the distinguished triangle 
\begin{equation}
\rmR\pi_\ast(\rmR q_{2!}K)_E \longrightarrow
\rmR\pi_\ast(\rmR q_{2!}\tl{L})_E \longrightarrow
\rmR\pi_\ast({}^\sfL G \vbar_{Y^\an\times\RR_t} [-1])_E 
\overset{+1}{\longrightarrow} 
\end{equation} 
by Lemma \ref{lemma:2-8} we obtain an equality
\begin{equation}
\zeta_{h,\infty}(i_! Sol_Y(\SM^\wedge))
=\zeta_{h,\infty}(i_!\rmR\pi_\ast(\rmR q_{2!}K)_E)\cdot 
\zeta_{h,\infty}(i_!\rmR\pi_\ast(\rmR q_{2!}\tl{L})_E)^{-1}.
\end{equation}
Since the cohomology sheaves of $\rmR\pi_\ast(\rmR q_{2!}\tl{L})_E$ 
are constant on the punctured disk $B(b_\infty)^\circ\subset Y^\an$ 
and have the trivial monodromy, for the study of the eigenvalues 
$\lambda\in\CC^\ast$ such that $\lambda\neq1$ 
in the monodromy at infinity of $\SM^\wedge$ 
it suffices to calculate the monodromy zeta function 
\begin{equation}
\zeta_{h,\infty}(i_!\rmR\pi_\ast(\rmR q_{2!}K)_E)(t)
\quad \in\CC(t)^\ast.
\end{equation} 
For this purpose, we fix $R\gg0$ such that 
\begin{equation}
 \gamma_R\coloneq \{w\in Y^\an=\CC\mid\abs{w}=R\}
\quad \subset B(b_\infty)^\circ
\end{equation}
and consider the monodromies of the local systems
\begin{equation}
H^j\rmR\pi_\ast(\rmR q_{2!} K)_E\vbar_{B(b_\infty)^\circ}
\quad (j\in\ZZ)
\end{equation}
along the loop $ \gamma_R\subset B(b_\infty)^\circ$ 
in the ``counterclockwise" direction.
For $t_0\gg0$ such that 
\begin{equation}
t_0>\max_{w\in  \gamma_R}\phi(w)
\end{equation} 
we define a subset $\tl{ \gamma_R}\simeq  \gamma_R$ of $E\subset Y^\an\times\RR_t$ by 
\begin{equation}
\tl{ \gamma_R}\coloneq\{(w,t_0)\mid w\in  \gamma_R\}
=\{(R e^{i\theta},t_0)\mid 0\leq\theta\leq2\pi\}
\end{equation}
and identify it with the loop $ \gamma_R$ naturally. 
Recall that the restrictions of the cohomology sheaves of 
$\rmR q_{2!} K$ to $E$ are locally constant. 
Then there exist isomorphisms
\begin{equation}
H^j\rmR\pi_\ast(\rmR q_{2!}K)_E\vbar_{ \gamma_R}\simeq
H^j(\rmR q_{2!}K)\vbar_{\tl{ \gamma_R}}
\quad (j\in\ZZ).
\end{equation}
For a point $(Re^{i\theta},t_0)\in\tl{ \gamma_R}$ 
($0\leq\theta\leq2\pi$) of the loop $\tl{ \gamma_R}$ let 
\begin{equation}
i_\theta\colon U^\an\simeq q_2^{-1}((Re^{i\theta},t_0))
\longhookrightarrow U^\an\times (Y^\an\times\RR_t)
\end{equation}
be the inclusion map and set 
$K_\theta\coloneq i_\theta^{-1}K 
\simeq K(Re^{i\theta},t_0) \in \BDC(\CC_{U^\an})$
so that we have an isomorphism
\begin{equation}
(\rmR q_{2!}K)_{(Re^{i\theta},t_0)}\simeq\rsect_c(U^\an; K_\theta).
\end{equation}
Then by our local description of the sheaves $K(w,t)$ in 
\eqref{locdescofK}, for sufficiently large $t_0 \gg 0$ 
the support of $K_\theta \simeq K(Re^{i\theta},t_0)$ 
is contained in the open subset 
\begin{equation}
B(a_1)^\circ\sqcup B(a_2)^\circ \sqcup\dots\sqcup B(a_l)^\circ
\sqcup B(a_\infty)^\circ \quad \subset U^\an.
\end{equation}
Moreover for $1\leq i\leq l$, if we 
choose sufficiently large $R \gg 0$ and $t_0 \gg 0$ 
so that we have $t_0 \gg R$, 
for any exponential factor $f \in N_{i}^{>0}$ 
of $\SM^{\an}$ at the point 
$a_{i} \in X^{\an}$ the level 
set $\{ z \in B(a_i)^\circ \ | \ 
\Re ( f^{R e^{i \theta}}(z) ) > t_0 \}$ of the 
the real part of the holomorphic 
function 
\begin{equation}
f^{R e^{i \theta}}(z)= z\cdot Re^{i\theta} -f(z)
\end{equation}
on the punctured disk 
$B(a_i)^\circ$ is perturbed only slightly 
as $\theta$ moves from $0$ to $2 \pi$. 
This implies that for any $1\leq i\leq l$ there exists 
$K_i\in \BDC(\CC_{B(a_i)^\circ})$ such that for any point 
$(Re^{i\theta},t_0)\in\tl{ \gamma_R}$ $(0\leq\theta\leq2\pi)$ 
we have an isomorphism 
\begin{equation}
\rsect_c(B(a_i)^\circ;K_i)\simto\rsect_c(B(a_i)^\circ;K_\theta)
\end{equation}
(see the construction of the enhanced sheaf $G$ on 
the punctured disk $B(a_i)^\circ$ in 
\cite[Section 2.5]{KT23}). 
It follows that for any $j \in \ZZ$ the automorphism of the 
cohomology group $H^j_c(B(a_i)^\circ;K_\theta)$ obtained as 
$\theta$ moves from $0$ to $2 \pi$ is equal to the identity. 
Hence for the study of 
the eigenvalues $\lambda\in \CC^\ast$
such that $\lambda\neq1$ we may replace $\rsect_c(U^\an;K_\theta)$
by $\rsect_c(B(a_\infty)^\circ;K_\theta)$ and 
have only to study how the 
cohomology groups $H^j_c(B(a_\infty)^\circ;K_\theta)$ 
change as $\theta$ moves from $0$ to $2 \pi$. 
By our assumption $\SM \in \Modleq(\SD_X)$, if we 
choose sufficiently large $R \gg 0$ and $t_0 \gg 0$ 
so that we have $t_0 \gg R$,  
for any exponential factor $f \in N_{\infty}^{>0}$ 
of $( \tl{\SM} )^{\an}$ at the point 
$a_{\infty}= \infty \in \var{X}^{\an}$ 
and any $0 \leq \theta \leq 2 \pi$ the level 
set $\{ z \in B(a_\infty)^\circ \ | \ 
\Re ( f^{R e^{i \theta}}(z) ) > t_0 \}$ of the 
the real part of the holomorphic 
function 
\begin{equation}
f^{R e^{i \theta}}(z)= z\cdot Re^{i\theta} -f(z)
\end{equation}
on the punctured disk $B(a_{\infty})^\circ$ 
is a slight perturbation of the contractible open subset 
$V_\theta\coloneq \{ z \in B(a_\infty)^\circ \ | \ 
\Re ( z\cdot Re^{i\theta} ) > t_0 \}$. 
After shrinking $V_\theta$ if necessary, we may 
assume also that for any $f \in N_{\infty}^{>0}$ we have 
the inclusion 
\begin{equation}
V_\theta \subset \{ z \in B(a_\infty)^\circ \ | \ 
\Re ( f^{R e^{i \theta}}(z) ) > t_0 \}. 
\end{equation}
This implies that there exists a morphism $L_{V_\theta} \longrightarrow 
K_\theta$ which induces an isomorphism
\begin{align}
\rsect_c(V_\theta;L) \simeq\rsect_c(B(a_\infty)^\circ;L_{V_\theta}) 
\simto \rsect_c(B(a_\infty)^\circ;K_\theta)
\end{align}
(see the construction of the enhanced sheaf $G$ on 
the punctured disk $B(a_{\infty})^\circ$ in 
\cite[Section 2.5]{KT23}). 
Moreover for any point $a\in V_\theta$ we have an isomorphism
\begin{equation}
\rsect_c(V_\theta;L)[2]\simeq L_a.
\end{equation}
Note that when we calculate the monodromy of 
the right hand side by a monodromy zeta function the shift 
$[2]$ in the left hand side can be neglected.
However note that if $\theta$ moves from $0$ to $2\pi$ 
the open subset $V_\theta\subset X^\an=\CC$ rotates in 
the ``clockwise" direction.
Then we immediately obtain the assertion. 
This completes the proof of Theorem \ref{MT-1}. 
\qed   \\

Finally, let us give some examples and compare Theorem \ref{MT-1} with results obtained by explicit calculations.
\begin{example}
\begin{enumerate}
\item [\rm (i)]
Let us consider the case where $\SM$ is Bessel's equation on $X=\CC_z$. 
Namely, for a complex number $\nu\in\CC$ we set
\begin{equation}
P\coloneq z^2\partial_z^2+z\partial_z+(z^2-\nu^2)
\end{equation}
and $\SM\coloneq \SD_X/\SD_XP\in\Modhol(\SD_X)$.
It is well-known that $z=\infty$ is an irregular singular point of $\tl{\SM}$ and the slope at infinity is 1.
Hence we have $\SM\in\Modleq(\SD_X)$.
Since $\SM$ is regular at the origin $z=0$ and $\ssupp(\tl{\SM})=\{0,\infty\}$, we obtain the monodromy matrix at infinity $A$ of $\SM$ by calculating the characteristic exponents of $P$ at the origin $z=0$.
Moreover it follows from an explicit calculation that the Fourier transform $P^\wedge$ of $P$ is regular at any point of $\var{Y}$ (cf. \cite[Corollary 4.14]{KT23}). 
Therefore we also obtain the monodromy matrix at infinity $A^\prime$ of $\SM^\wedge$ by the characteristic exponents of $P^\wedge$ at $w=\infty$.
Explicitly, we have
\begin{align}
A&=
\begin{pmatrix}
e^{2\pi\sqrt{-1}\nu} &  0 \\
0 & e^{-2\pi\sqrt{-1}\nu} \\
\end{pmatrix}
, \\
A^\prime&=
\begin{pmatrix}
e^{2\pi\sqrt{-1}\nu} &  0 \\
0 & e^{-2\pi\sqrt{-1}\nu} \\
\end{pmatrix},
\end{align}
as expected.

\item [\rm (ii)]
Let us consider the case where $\SM$ is 
the Gauss hypergeometric equation on $X=\CC_z$. 
Namely, for complex numbers $\alpha,\beta,\gamma\in\CC$ we set
\begin{equation}
P\coloneq z(1-z)
\partial_z^2+\{\gamma-(\alpha+\beta+1)z\}\partial_z-\alpha\beta
\end{equation}
and $\SM\coloneq\SD_X/\SD_XP\in\Modhol(\SD_X)$.
In this case, we have $\tl{\SM}\in\Modrh(\SD_{\var{X}})$ and $\ssupp(\tl{\SM})=\{0,1,\infty\}$.
Hence we can calculate the monodromy matrix at infinity $A$ of $\SM$ by the characteristic exponents of $P$ at $z=\infty$.
Explicitly, we have 
\begin{equation}
A=
\begin{pmatrix}
e^{-2\pi\sqrt{-1}\alpha} &  0 \\
0 & e^{-2\pi\sqrt{-1}\beta} \\
\end{pmatrix}
.
\end{equation}
On the other hand, the restriction of 
$\SM^\wedge$ to $Y \setminus \{ 0 \}$ 
is an integrable connection (see \cite[Theorem 1.2]{IT20b}). 
Moreover the meromorphic connection $(\SM^\wedge) (\ast \{ 0 \})$ 
is regular at the origin 
$0 \in Y$ and $\SM^\wedge\in\Modleq(\SD_Y)$ (see 
\cite[Theorems 4.4 and 4.6]{IT20a}). It follows that 
the monodromy at infinity of $\SM^\wedge$ is equal to 
the one around the origin $0 \in Y$. 
Similarly to the case (i), the monodromy matrix 
at infinity of $\SM^\wedge$ is 
\begin{equation}
A^\prime=
\begin{pmatrix}
e^{2\pi\sqrt{-1}\alpha} &  0 \\
0 & e^{2\pi\sqrt{-1}\beta} \\
\end{pmatrix}.
\end{equation}   

\item [\rm (iii)]
Let us consider the case where $\SM$ is 
Katz's hypergeometric equation on $X=\CC_z$. 
Namely, for complex parameters $\alpha =( \alpha_1, \ldots, \alpha_n) 
\in \CC^n$, $\beta =( \beta_1, \ldots, \beta_m) 
\in \CC^m$ and $\gamma \in \CC^*$ we set
\begin{equation}
P \coloneq \gamma \cdot \prod_{i=1}^n (z \partial_z- \alpha_i) - 
z \cdot \prod_{j=1}^m (z \partial_z- \beta_j)
\end{equation}
and $\SM\coloneq\SD_X/\SD_X P \in\Modhol (\SD_X)$. 
This $\SD$-module has been widely studied from 
various points of view by several authors (see \cite{BHHS23}, \cite{CDS21}, \cite{DM89}, \cite{Fed18}, \cite{Hie22} and \cite{Kat90}). 
Assume that $d \coloneq n-m >0$. Then $z= \infty$ is an 
irregular singular point of $\tl{\SM}$ and the slopes 
at infinity are $0$ and $1/d \leq 1$. Hence we have 
$\SM \in\Modleq(\SD_X)$. Moreover $\SM$ is regular at 
the origin $z=0$ and ${\rm sing.supp} ( \tl{\SM} )= \{ 0, \infty \}$. 
We thus can calculate the monodromy at infinity of $\SM$ 
by the characteristic exponents $\alpha_1, \ldots, \alpha_n$ 
of $P$ at the origin $z=0$. Explicitly, its eigenvalues are 
$\exp (2 \pi i \alpha_i)$ ($1 \leq i \leq n$). 
On the other hand, for the Fourier transform $P^{\wedge}$ of 
$P$ we have 
\begin{equation}
P^{\wedge} = \gamma \cdot \prod_{i=1}^n (-w \partial_w- 1- \alpha_i)  
+ \partial_w \cdot \prod_{j=1}^m (-w \partial_w -1- \beta_j)
\end{equation}
and hence $P^{\wedge}$ is regular at the point $w= \infty$ (see also \cite[Theorems 1.1 and 4.16]{KT23}). 
Then by calculating the characteristic exponents 
of $P^{\wedge}$ at $w= \infty$, we can easily see 
that the eigenvalues of the  monodromy at infinity of $\SM^{\wedge}$ 
are $\exp (-2 \pi i \alpha_i)$ ($1 \leq i \leq n$), as expected. 
\end{enumerate}
\end{example}

\section{Monodromies at infinity of 
Fourier transforms of 
regular holonomic D-modules}\label{dim-N}

For an algebraic regular holonomic $\SD_X$-modules $\SM$ on $X=\CC_z^N$, 
let us consider the monodromies at infinity of $Sol_Y(\SM^\wedge)$. 
Let $\Omega\subset Y=\CC_w^N$ 
be the maximal open subset of $Y$ such that the restriction 
\begin{equation}
\Char(\SM)\cap q^{-1}(\Omega)\longrightarrow \Omega
\end{equation} 
of the projection $q\colon T^\ast X\simeq X\times Y\longrightarrow Y$ 
is an unramified (finite) covering. 
Then by the conicness of $\Char(\SM)\subset T^\ast X$ 
we can easily see that $\Omega$ is 
also a conic open subset of $Y=\CC_w^N$.
By \cite[Corollary 4.5]{IT20a} the Fourier transform $\SM^\wedge$ of $\SM$ is 
an integrable connection on $\Omega$. 
This implies that the restriction of 
its solution complex $Sol_Y(\SM^\wedge)\in\Dbc(\CC_{Y^\an})$ to 
$\Omega^\an\subset Y^\an$ is a local system.
Take a point $w_0 \in\Omega^\an$ such that 
$w_0\neq0$ and set 
\begin{equation}
\LL\coloneq\{\tau w_0 \mid \tau\in\CC\} 
\quad \subset Y^\an\simeq\CC^N.
\end{equation}
Assume that $\Omega \not= Y$. 
Then we have $\LL\simeq\CC$ and 
$\LL\cap\Omega=\LL\bs\{0\}\simeq\CC\bs\{0\}$.
Hence for $R>0$ we can consider the monodromy of 
the local system $Sol_Y(\SM^\wedge)\vbar_{\Omega^\an}$ along the loop 
\begin{equation}
C_R\coloneq\{Re^{i\theta}w_0\mid 0\leq\theta\leq2\pi\} 
\quad \subset\LL\bs\{0\}
\end{equation}
in the complex line $\LL\subset Y^\an$ as follows. 
Take a point $b \in C_R$ of $C_R$ and let $\Psi_b$ be 
the $\CC$-linear automorphism of the stalk 
\begin{equation}
Sol_Y(\SM^{\wedge})_b \simeq
\shom_{\SD_{Y^\an}}(( \SM^{\wedge})^\an,\SO_{Y^\an})_b
\end{equation}
at $b$ induced by the analytic continuations along the loop $C_R$ 
in the ``counterclockwise" direction. 
We denote by $\rank\SM^{\wedge}$ the generic rank of 
the Fourier transform $\SM^{\wedge}$ of 
the holonomic $\SD_X$-module $\SM$ so that we have 
$Sol_Y(\SM^{\wedge} )_b \simeq\CC^{\rank{\SM^{\wedge}}}$. 
Then the square matrix $A$ of size $\rank\SM^{\wedge}$ representing 
the automorphism $\Psi_b$ is defined up to conjugacy. 
We call it the monodromy (matrix) at infinity of $\SM^{\wedge}$.
It is clear that the conjugacy class of $A$ in 
$\GL_{\rank\SM^{\wedge}}(\CC)$ does not depend on 
the choices of $R >0$ and $b \in C_R$. Moreover 
it does not depend on the choice of the point 
$w_0 \in\Omega^\an$ such that 
$w_0\neq0$. For a non-zero complex number $\lambda\in\CC^\ast=\CC\bs\{0\}$ 
we denote by $\mu (\SM^{\wedge},\lambda)$ the multiplicity of the eigenvalue 
$\lambda$ in the monodromy at infinity of $\SM^{\wedge}$. 
Let $i\colon\CC\longhookrightarrow \PP^1$ be 
the inclusion map and $h$ a holomorphic coordinate of $\PP^1$ defined on 
a neighborhood of the point $\infty\in\PP^1$ in 
$\PP^1$ such that $h^{-1}(0)=\{\infty\}$. We define a 
morphism $g\colon X^\an=\CC^N\longrightarrow\CC$ by 
\begin{equation}
g\colon X^\an=\CC^N\longrightarrow\CC
\quad (z\longmapsto\innpro{z,w_0})
\end{equation}
and set $\SF \coloneq Sol_X(\SM)\in \Dbc(\CC_{X^\an})$. 

\begin{theorem}\label{MT-2} 
For any non-zero complex number 
$\lambda\in\CC^\ast=\CC\bs\{0\}$ such that $\lambda\neq1$ 
the multiplicity $\mu (\SM^{\wedge},\lambda)$ of it 
in the the monodromy at infinity of $\SM^{\wedge}$ 
is equal to that of 
the factor $t- \lambda$ (resp. $1-\lambda t$) in the 
rarional function 
$\Bigl\{ \tl{\zeta}_{h,\infty}(i_! \rmR g_!\SF)(t) \Bigr\}^{(-1)^{N+1}} \in\CC(t)^\ast$ 
(resp. $\Bigl\{ \zeta_{h,\infty}(i_! \rmR g_!\SF)(t) \Bigr\}^{(-1)^{N+1}} 
\in\CC(t)^\ast$). 
\end{theorem}

\begin{proof}
Let $\var{X}\simeq\PP^N$ be 
the projective compactification of $X=\CC^N$ and 
$i_X\colon X\hookrightarrow\var{X}$ the inclusion map. 
We set $\tl{\SM}\coloneq\bfD i_{X\ast}\SM\simeq i_{X\ast}\SM\in
\Modhol(\SD_{\var{X}})$.
By the inclusion map $i_{X^{\an}}\colon X^{\an} \hookrightarrow \var{X}^{\an}$ 
we define an $\RR$-constructible enhanced sheaf 
$G \in \BECrc(\CC_{\var{X}^\an})$ on $\var{X}^\an \simeq \PP^N$ by 
\begin{equation}
G \coloneq \epsilon \Bigl( (i_{X^{\an}})_! \SF \Bigr) 
\simeq\CC_{\{t\geq0\}}\otimes \pi^{-1} (i_{X^{\an}})_! \SF \quad 
\in\BECrc(\CC_{\var{X}^\an}). 
\end{equation} 
Then by Proposition \ref{prop-K2} (v) we obtain an isomorphism 
\begin{equation}
\SolE_{\var{X}}(\tl{\SM})\simeq \CCE_{\var{X}^\an}\Potimes G.
\end{equation}
Moreover by the proof of \cite[Theorem 4.4]{IT20a}, 
for the (enhanced) Fourier-Sato transform 
${}^\sfL G \in\BECrc(\CC_{\var{Y}^\an})$ of $G$ and 
a point $(w,t)\in C_R \times\RR 
\subset ( \LL\bs\{0\} ) \times \RR$ there exists an isomorphism
\begin{equation}
{}^\sfL G_{(w,t)}\simeq
\rsect_c(\{z\in X^\an\mid \Re\innpro{z,w}\leq t\}; \SF)[N].
\end{equation} 
For a sufficiently large $t_0 \gg 0$ we set also
\begin{equation}
\tl{C_R}\coloneq \{ (w,t_0) \mid w \in C_R \} = 
\{(Re^{i\theta}w_0,t_0)\mid 0\leq\theta\leq2\pi\} 
\quad \subset ( \LL\bs\{0\} ) \times \RR 
\subset Y^\an\times\RR
\end{equation}
and identify it with the loop $C_R\subset Y^\an$ naturally.
Then by the proof of \cite[Corollary 4.5]{IT20a} we obtain an isomorphism 
\begin{equation}
Sol_Y(\SM^\wedge)\vbar_{C_R}\simeq {}^\sfL G\vbar_{\tl{C_R}}.
\end{equation}
Hence, for the study of the monodromy of $Sol_Y(\SM^\wedge)$ along 
the loop $C_R$, it suffices to study that of ${}^\sfL G$ along $\tl{C_R}$.
For $0\leq\theta\leq2\pi$ let us set
\begin{equation}
Z_\theta\coloneq\{z\in X^\an\mid \Re\innpro{z,Re^{i\theta}w_0}\leq t_0\}
\quad \subset X^\an.
\end{equation}
Then for any point $(Re^{i\theta}w_0,t_0)\in\tl{C_R}$ of 
the loop $\tl{C_R}$ we obtain isomorphisms 
\begin{align}
{}^\sfL G_{(Re^{i\theta}w_0, t_0)} \simeq\rsect_c(Z_\theta;\SF)[N] 
\simeq\rsect_c(X^\an;\SF_{Z_\theta})[N].
\end{align}
For $0\leq\theta\leq2\pi$ we set also 
\begin{equation}
U_\theta\coloneq X^\an\bs Z_\theta
=\{z\in X^\an\mid \Re\innpro{z,Re^{i\theta}w_0}>t_0\}
\end{equation}
so that we have distinguished triangles 
\begin{equation}
\SF_{U_\theta}\longrightarrow\SF
\longrightarrow\SF_{Z_\theta}\overset{+1}{\longrightarrow}
\end{equation}
and 
\begin{equation}
\SF [N] \longrightarrow\SF_{Z_\theta} [N]
\longrightarrow \SF_{U_\theta} [N+1] 
\overset{+1}{\longrightarrow}. 
\end{equation}
Hence as in the proof of Theorem \ref{MT-1}, 
by the theory of monodromy zeta functions we can show that 
it suffices to study the automorphism of 
$\rsect_c(X^\an;\SF_{U_\theta} [N+1])$ obtained as 
$\theta$ moves from $0$ to $2\pi$.
For this purpose, for $0\leq\theta\leq2\pi$ we define 
an open subset $V_\theta\subset\CC$ of $\CC$ by
\begin{equation}
V_\theta\coloneq\{\tau\in\CC\mid\Re(Re^{i\theta} \tau )>t_0\}
\quad \subset\CC
\end{equation}
and consider the algebraic map 
\begin{equation}
g\colon X^\an=\CC^N\longrightarrow\CC
\quad (z\longmapsto\innpro{z,w_0}).
\end{equation}
Then for any $0\leq\theta\leq2\pi$ we have 
\begin{equation}
U_\theta=\{z\in X^\an\mid g(z)\in V_\theta\}
=g^{-1}(V_\theta)
\end{equation}
and hence there exist isomorphisms
\begin{align}
\rsect_c(X^\an;\SF_{U_\theta})
&\simeq\rsect_c(\CC;\rmR g_!(\SF_{U_\theta})) \notag \\ 
&\simeq\rsect_c(\CC;\rmR g_!(\CC_{U_\theta}\otimes\SF)) \notag \\
&\simeq\rsect_c(\CC;\rmR g_!(g^{-1}\CC_{V_\theta}\otimes\SF)) \notag \\
&\simeq\rsect_c(\CC;\CC_{V_\theta}\otimes\rmR g_!\SF) \notag \\
&\simeq\rsect_c({V_\theta};\rmR g_!\SF).
\end{align}
Now note that by the algebraicity of $\SM$ 
the solution complex $\SF=Sol_X(\SM)\in\Dbc(\CC_{X^\an})$ is 
algebraically constructible.
Since the morphism $g\colon X^\an=\CC^N\longrightarrow\CC$ is algebraic, 
this implies that $\rmR g_!\SF\in\Dbc(\CC_{\CC})$ is 
also algebraically constructible.
We thus can take the sufficiently large number $t_0>0$ 
so that for any $0\leq\theta\leq2\pi$ and $j\in\ZZ$ 
the restriction of 
$H^j\rmR g_!\SF$ to the open subset $V_\theta\subset\CC$ 
is a local system.
So, after replacing $t_0>0$ if necessary, 
we may assume that for any $0\leq\theta\leq2\pi$ and 
any point $b\in V_\theta$ in $V_\theta$ we have an isomorphism
\begin{equation}
\rsect_c(V_\theta;\rmR g_!\SF)[2]\simeq(\rmR g_!\SF)_b.
\end{equation} 
Note that when we calculate the 
monodromy of the right hand side by a monodromy 
zeta function 
the shift $[2]$ on the left hand side 
can be neglected. 
Hence it suffices to study the monodromy at infinity of 
$\rmR g_!\SF\in\Dbc(\CC_\CC)$.
However note that if $\theta$ moves from $0$ to $2\pi$ 
the open subset $V_\theta\subset\CC$ rotates 
in the ``clockwise" direction. This rotation is in 
the ``counterclockwise" direction when we view it 
from the point $\infty \in \PP^1$. Then the assertion 
immediately follows. 
\end{proof} 

From now, we shall explain how the monodromy zeta function 
$\zeta_{h,\infty}(i_! \rmR g_!\SF)(t) \in\CC(t)^\ast$ in 
Theorem \ref{MT-2} can be calculated. 
Let $H\coloneq\var{X}\bs X\simeq\PP^{N-1}$ be 
the hyperplane at infinity of $\var{X} \simeq \PP^N$ and 
$\tl{g}$ the meromorphic function on $\var{X}^\an \simeq \PP^N$ 
obtained by extending $g:X^{\an} \longrightarrow \CC$ to $\var{X}^\an$. 
Then it has a pole along $H^\an\subset\var{X}^\an$.
Moreover, the closure $\var{g^{-1}(0)}\subset\var{X}^\an$ of 
$g^{-1}(0)\subset X^\an$ in $\var{X}^\an$ 
intersects $H^\an$ transversally and 
the set of the points of 
indeterminacy of the meromorphic function $\tl{g}$ is 
equal to the codimension two 
complex submanifold 
$\var{g^{-1}(0)}\cap H^\an \subset \var{X}^\an$ of $\var{X}^\an$.
Let $\nu\colon(\var{X}^\an)^\sim\longtwoheadrightarrow\var{X}^\an$ 
be the blow-up of $\var{X}^\an$ along 
$\var{g^{-1}(0)}\cap H^\an$ and 
$\iota\colon X^\an\longhookrightarrow(\var{X}^\an)^\sim$ 
the inclusion map. Then the meromorphic function $\tl{g} \circ \nu$ 
has no point of indeterminacy on the whole 
$(\var{X}^\an)^\sim$ and hence 
there exists a proper holomorphic map 
$g^\prime\colon(\var{X}^\an)^\sim\longrightarrow\PP^1$ such that 
for the inclusion map $i\colon\CC\longhookrightarrow \PP^1$ 
we have a commutative diagram
\begin{equation}
\vcenter{\xymatrix@M=7pt{
X^\an \ar@{^{(}->}[r]^-{\iota}
\ar@{->}[d]_-{g}  &
(\var{X}^\an)^\sim \ar@{->}[d]^-{g^\prime} \\
\CC \ar@{^{(}->}[r]_-{i} &
\PP^1.       
}}
\end{equation}
We thus obtain an isomorphism
\begin{equation}\label{eq:iRg}
i_! \rmR g_!\SF\simeq\rmR g^\prime_\ast\iota_!\SF.
\end{equation}
Then it follows from the isomorphism (\ref{eq:iRg}) and 
Theorem \ref{prp:2-99} that there exists an equality
\begin{equation}
\zeta_{h,\infty}(i_! \rmR g_!\SF)
= \int_{(h \circ g^\prime)^{-1}(0)}\zeta_{h\circ g^\prime}(\iota_!\SF)
=\int_{(g^\prime)^{-1}(\infty)}\zeta_{h\circ g^\prime}(\iota_!\SF)
\end{equation}
in the multiplicative group $\CC(t)^\ast$, where
\begin{equation}
\int_{(g^\prime)^{-1}(\infty)}
\colon\CF_{\CC(t)^\ast}((g^\prime)^{-1}(\infty))
\longrightarrow \CC(t)^\ast
\end{equation}
is the topological integral of 
$\CC(t)^\ast$-valued constructible functions over 
the analytic subset $(g^\prime)^{-1}(\infty)\subset(\var{X}^\an)^\sim$.
Note that $(g^\prime)^{-1}(\infty)$ is nothing but the 
proper transform of $H^{\an}$ in the blow-up $(\var{X}^\an)^\sim$. 
In this situation, we can calculate 
$\dint_{(g^\prime)^{-1}(\infty)}\zeta_{h\circ g^\prime}(\iota_!\SF)
\in\CC(t)^\ast$ by the methods in 
Section \ref{sec:zeta} as follows.  
Set $Z\coloneq(\var{X}^\an)^\sim$ and $\tl{\SF}\coloneq\iota_!\SF$.
Let $\calS$ be a stratification of 
$Z$ adapted to the constructible sheaf $\tl{\SF}$ and 
for $k \geq 0$ set 
\begin{equation}
Z_k\coloneq\bigsqcup_{S\in\calS\colon \dim S \leq k} S 
\quad \subset Z.
\end{equation}
Then $Z_k$ are (closed) analytic subsets of 
$Z$ and for any $k\geq 1$ we have a distinguished triangle
\begin{equation}
\tl{\SF}_{Z_k\bs Z_{k-1}}\longrightarrow
\tl{\SF}_{Z_k}\longrightarrow
\tl{\SF}_{Z_{k-1}}\overset{+1}{\longrightarrow}.
\end{equation}
Decomposing the support of $\tl{\SF}$ in this way and 
using also truncations and Lemma \ref{lemma:2-8}, 
we may assume that $\tl{\SF}=i_{S!}L$ for 
a local system $L$ on a stratum $S\in\calS$, 
where $i_S\colon S\longhookrightarrow Z$ is the inclusion map.
Since $\tl{\SF}\vbar_{Z\bs\iota(X^\an)}\simeq0$, 
the condition $L\not\simeq0$ implies that $S$ is 
contained in the open subset $\iota(X^\an)\subset Z$.
Hence we may assume also that $S\subset\iota(X^\an)$.
This in particular 
implies that we have $S\cap(g^\prime)^{-1}(\infty)=\emptyset$.
Let $\var{S}\subset Z$ be the closure of $S$ in $Z$.
Then there exists a proper morphism $\rho\colon T\longrightarrow Z$ of 
a smooth complex manifold $T$ such that 
$\rho(T)=\var{S}$ inducing an isomorphism 
$\rho^{-1}(S)\simto S$ and 
$D\coloneq\rho^{-1}(\var{S}\bs S)\subset T$ is 
a normal crossing divisor in $T$.  
Let $j_S\colon S\simeq \rho^{-1}(S)\longhookrightarrow T$ 
be the inclusion map.
Then there exist isomorphisms 
\begin{equation}
\tl{\SF}\simeq i_{S!}L
\simeq \rmR \rho_! j_{S!}L 
\simeq \rmR \rho_\ast j_{S!}L.
\end{equation}
By Theorem \ref{prp:2-99} we thus obtain
\begin{align}
\int_{(g^\prime)^{-1}(\infty)}
\zeta_{h\circ g^\prime}(\tl{\SF})
&=\int_{(h\circ g^\prime)^{-1}(0)}
\zeta_{h\circ g^\prime}(i_{S!}L) \notag \\
&=\int_{(h\circ g^\prime\circ\rho)^{-1}(0)}
\zeta_{h\circ g^\prime\circ\rho}(j_{S!}L).
\end{align}
Moreover, by the condition 
$S\cap(g^\prime)^{-1}(\infty)=S\cap(h\circ g^\prime)^{-1}(0)=\emptyset$ 
for the holomorphic function 
$f_S\coloneq h\circ g^\prime\circ\rho\colon T
\longrightarrow\CC$ on $T$ 
we have $j_S(S)\cap f_S^{-1}(0)=\emptyset$.
This implies that $f_S^{-1}(0)\subset T$ is 
contained in the normal crossing divisor $D\subset T$.
Namely $f_S^{-1}(0)$ is a union of some irreducible components of $D$.
In this situation, we can apply 
Lemmas \ref{lem:2-99} and \ref{lem:2-999} 
to calculate the $\CC (t)^*$-valued constructible function
\begin{equation}
\zeta_{h\circ g^\prime\circ\rho}(j_{S!}L)
=\zeta_{f_S}(j_{S!}L)
\quad \in \CF_{\CC(t)^\ast}(f_S^{-1}(0))
\end{equation}
at each point of $f_S^{-1}(0)\subset T$.
Then by Corollary \ref{cpr:2-99} we can calculate 
the monodromy zeta function 
$\zeta_{h,\infty}(i_! \rmR g_!\SF)(t)\in \CC(t)^\ast$ 
in the general case. 

As a special case, we have the following result 
for the Fourier transforms of some regular meromorphic connections.

\begin{theorem}\label{MT-3}
Let $f(z)\in\CC[z_1,z_2,\ldots,z_N]$ be a polynomial of 
degree $d>0$ such that the hypersurface 
$\{f_d=0\}\subset H^\an\simeq\PP^{N-1}$ defined by 
its top degree part $f_d(z)$ is smooth and 
$j\colon X^\an\bs f^{-1}(0)\longhookrightarrow X^\an$ the inclusion map.
For a complex number $\alpha\in\CC$ and the rank one local system 
$L\coloneq\CC_{X^\an\bs f^{-1}(0)}f^\alpha(z)$ 
on $X^\an\bs f^{-1}(0)$ associated to it, 
let $\SM\in\Modrh(\SD_X)$ be the 
algebraic regular meromorphic connection on 
$X$ such that $Sol_X(\SM)\simeq j_!L$ 
(see \cite[Theorem 5.2.24 and Corollary 5.3.10]{HTT08}).
Then for $\SF\coloneq j_!L\in\Dbc(\CC_{X^\an})$ we have 
\begin{equation}
\zeta_{h,\infty}(i_!\rmR g_!\SF)(t)
=(1-te^{-2\pi\sqrt{-1}d\alpha})^{(1-d)^{N-1}}.
\end{equation}
Moreover, for any non-zero complex number 
$\lambda\in\CC^\ast$ such that $\lambda\neq1$ and 
the multiplicity $\mu(\SM^\wedge,\lambda)$ of it 
in the monodromy at infinity of $\SM^\wedge$ we have 
\begin{align}
\mu(\SM^\wedge,\lambda) = 
\begin{cases}
\ (d-1)^{N-1} & ( \lambda = e^{{-2\pi\sqrt{-1}d\alpha}}), \\
& \\
\ 0 & (\textit{otherwise}).
\end{cases}
\end{align}     
\end{theorem} 

\begin{proof}
Note that for the closure 
$\var{f^{-1}(0)}\subset\var{X}^\an$ of 
$f^{-1}(0)\subset X^\an$ in $\var{X}^\an$ we have 
$\var{f^{-1}(0)} \cap H^\an =
\{f_d=0\}$. Then the smoothness of 
$\{f_d=0\}\subset H^\an\simeq\PP^{N-1}$ is equivalent to 
the condition that the hypersurface $\var{f^{-1}(0)}\subset\var{X}^\an$ 
is smooth in a neighborhood of $H^{\an}$ in $\var{X}^\an$ 
and intersects $H^\an$ transversally. The same is 
true also for the hypersurface $\var{g^{-1}(0)}\subset\var{X}^\an$. 
Moreover, by our definition of $\Omega \subset Y= \CC^N$ and the 
condition $w_0 \in \Omega^{\an}$, we can easily show that 
the divisor 
\begin{equation}
\var{f^{-1}(0)} \cup \var{g^{-1}(0)} \cup H^{\an} 
\qquad \subset \var{X}^\an
\end{equation}
is normal crossing in a neighborhood of $H^{\an}$ in $\var{X}^\an$. 
We denote by $E$ the exceptional divisor 
$\nu^{-1}(\var{g^{-1}(0)}\cap H^\an) \subset Z=(\var{X}^\an)^\sim$ of the 
blow-up $\nu\colon Z=(\var{X}^\an)^\sim\longtwoheadrightarrow\var{X}^\an$ 
of $\var{X}^\an$ along $\var{g^{-1}(0)}\cap H^\an$. 
Recall that $(g^\prime)^{-1}(\infty)\simeq\PP^{N-1}$ is 
the proper transform of $H^\an$ in the blow-up $Z=(\var{X}^\an)^\sim$ 
and denote it by $\tl{H}$. Let $( \var{f^{-1}(0)})^{\sim} \subset Z$ be 
the proper transform of $\var{f^{-1}(0)}$ in $Z$. Then the divisor 
\begin{equation}
( \var{f^{-1}(0)})^{\sim} \cup E \cup \tl{H}
\qquad \subset Z=(\var{X}^\an)^\sim
\end{equation}
is normal crossing in a neighborhood of $\tl{H}$ in $Z$. 
Now let us consider the following two hypersurfaces in $\tl{H}\subset Z$:
\begin{equation}
D_1\coloneq  ( \var{f^{-1}(0)})^{\sim} \cap \tl{H}, \quad
D_2\coloneq E \cap \tl{H}.
\end{equation}
Then $D_1$ and $D_2$ are smooth and intersect transversally in 
$\tl{H} \simeq \PP^{N-1}$. 
Note that the meromorphic extension of $f$
to $Z$ has a pole of order $d>0$ along
the divisor $\tl{H} \subset Z$.
Hence for a point $p\in\tl{H}=(g^\prime)^{-1}(\infty)$ 
by Lemmas \ref{lem:2-99} and \ref{lem:2-999} we obtain 
\begin{align}
\zeta_{h\circ g^\prime,p}(\iota_!\SF)(t) =
\begin{cases}
\ 1-te^{-2\pi\sqrt{-1}d\alpha} & (p\in\tl{H}\bs(D_1\cup D_2)), \\
& \\
\ 1 & (\textit{otherwise}).
\end{cases}
\end{align}      
Moreover by Proposition \ref{prp:2-999} we have $\chi(\tl{H})=N$ and 
\begin{align}
\chi(D_1)&=\frac{1}{d}\{(1-d)^N-1\}+N, \\
\chi(D_2)&=N-1, \\
\chi(D_1\cap D_2)&=\frac{1}{d}\{(1-d)^{N-1}-1\}+(N-1).
\end{align}
Then by Lemma \ref{lem:2-str-1} we obtain 
\begin{align}
\chi(\tl{H}\bs(D_1\cup D_2)) &=\chi_\rmc(\tl{H}\bs(D_1\cup D_2)) \\
&=\chi(\tl{H})-\chi(D_1)-\chi(D_2)+\chi(D_1\cap D_2) \\
&=(1-d)^{N-1}
\end{align}
and hence the first assertion 
\begin{align}
\zeta_{h,\infty}(i_!\rmR g_!\SF)(t)&=
\int_{(g^\prime)^{-1}(\infty)}\zeta_{h\circ g^\prime}(\iota_!\SF) \\
&=(1-te^{-2\pi\sqrt{-1}d\alpha})^{\chi(\tl{H}\bs(D_1\cup D_2))} \\
&=(1-te^{-2\pi\sqrt{-1}d\alpha})^{(1-d)^{N-1}}.
\end{align}
By Theorem \ref{MT-2} the second assertion immediately follows from the first one.
\end{proof}

Finally, let us give some examples to which Theorem \ref{MT-3} does not apply.
\begin{example}
\begin{enumerate}
\item [\rm (i)]
Let $N=2$ and $f(z_1,z_2)=z_1^3+z_1^2z_2-z_2\in\CC[z_1,z_2]$.
We consider the regular meromorphic connection 
$\SM\in\Modrh(\SD_X)$ associated to $f$ and $\alpha\in\CC$ (as in Theorem \ref{MT-3}) 
and the solution complex $\SF\coloneq Sol_X(\SM)\in\Dbc(\CC_{X^\an})$ of it. 
However the hypersurface $\{f_d=0\}\subset H^\an\simeq\PP^1$ is not smooth 
and we cannot apply Theorem \ref{MT-3} to this case. 
By an explicit calculation we see that the covering degree of 
$\Char(\SM)\cap q^{-1}(\Omega)\longrightarrow\Omega$ is 4.
In what follows, we use  the notations in the proof of Theorem \ref{MT-3}.
Note that $\var{f^{-1}(0)}\cap H^\an$ and $\var{g^{-1}(0)}\cap H^\an$ do not intersect.
Recall that $\nu\colon Z=(\var{X}^\an)^\sim\longtwoheadrightarrow\var{X}^\an$ is the blow-up of $\var{X}^\an$ along $\var{g^{-1}(0)}\cap H^\an$.
Then $D_1=(\var{f^{-1}(0)})^\sim\cap \tl{H}$ consists of two points 
and its singular locus consists of a single point that we denote by $P$. 
Let $\rho\colon T\longtwoheadrightarrow Z=(\var{X}^\an)^\sim$ 
be the blow-up of $Z$ along $P$ and $\iota^\prime\colon X^\an\longhookrightarrow T$ the inclusion map.
Denote by $E_P$ the exceptional divisor of it 
and by $(\var{f^{-1}(0)})^\approx$ and $\tl{H}^\sim$ 
the proper transforms of $({f^{-1}(0)})^\sim$ and $\tl{H}$ in $T$, 
respectively (see Figure \ref{fig:bl}).
\begin{figure}[b]
\centering
\begin{tikzpicture}[scale=0.5]
\draw (0,0) node{$\longleftarrow$} node[below]{$\rho$};
\begin{scope}[xshift=-6cm]
\begin{scope}
\coordinate (O) at(0,0);
\draw[name path=S1] (O) circle[radius=2.5];
\draw[name path=L1] (0,-0.75)--(-2,-3.2); 
\path [name intersections={of=S1 and L1,by=P0}];
\fill[white] (P0) circle[radius=0.105];
\draw (P0) circle[radius=0.105];
\begin{scope}[xshift=-1cm,yshift=2.4cm,scale=1.75,rotate=-7.5]
\draw[name path=C1] (0,-1)..controls (0.6,1.25) and (1.6,1.1)..(1.5,0);
\draw[name path=C2] (0.17,1)..controls (0.6,-1.25) and (1.4,-0.8)..(1.5,0);
\path [name intersections={of=C1 and C2,by=P1}];
\draw (-0.05,0.25) node{$P$};
\fill[white] (P1) circle[radius=0.06];
\draw (P1) circle[radius=0.06];
\path [name intersections={of=S1 and C2,by=P2}];
\fill[white] (P2) circle[radius=0.06];
\draw (P2) circle[radius=0.06];
\end{scope}
\draw (1.9,-2.4) node{$\tl{H}$};
\draw (-1,-2.9) node{$D_2$};
\draw (2.85,4.1) node{$(\var{f^{-1}(0)})^\sim$};
\draw (-2.5,-3.6) node{$E$};
\node[draw,outer sep=0pt] at(-3.5,3.5) {$Z$};
\end{scope}
\end{scope}
\begin{scope}[xshift=6cm]
\begin{scope}
\coordinate (O) at(0,0);
\draw[name path=S1] (O) circle[radius=2.5];
\draw[name path=L1] (0,-0.75)--(-2,-3.2); 
\path [name intersections={of=S1 and L1,by=P0}];
\fill[white] (P0) circle[radius=0.105];
\draw (P0) circle[radius=0.105];
\begin{scope}[xshift=-1cm,yshift=2.4cm,scale=1.75,rotate=-7.5]
\draw[name path=C1,white] (0,-1)..controls (0.6,1.25) and (1.6,1.1)..(1.5,0);
\draw[name path=C2,white] (0.17,1)..controls (0.6,-1.25) and (1.4,-0.8)..(1.5,0);
\draw[name path=L2] (0.15,-1)--(0.65,1.2);
\draw[name path=C3] (0.17,0.8)..controls (0.6,0.75) and (1.6,0.7)..(1.5,0);
\draw[name path=C4] (0,-0.6)..controls (0.6,-0.7) and (1.4,-0.6)..(1.5,0);
\path [name intersections={of=C1 and C2,by=P1}];
\fill[white] (P1) circle[radius=0.06];
\draw (P1) circle[radius=0.06];
\path [name intersections={of=S1 and C2,by=P2}];
\fill[white] (P2) circle[radius=0.06];
\draw (P2) circle[radius=0.06];
\path [name intersections={of=L2 and C3,by=P3}];
\fill[white] (P3) circle[radius=0.06];
\draw (P3) circle[radius=0.06];
\path [name intersections={of=L2 and C4,by=P4}];
\fill[white] (P4) circle[radius=0.06];
\draw (P4) circle[radius=0.06];
\draw (0.2,-1.3) node{$E_P$};
\end{scope}
\draw (2.1,-2.4) node{$\tl{H}^\sim$};
\draw (-1,-2.9) node{$D_2$};
\draw (2.7,3.8) node{$(\var{f^{-1}(0)})^\approx$};
\draw (-2.5,-3.6) node{$E$};
\node[draw,outer sep=0pt] at(-3.5,3.5) {$T$};
\end{scope}
\end{scope}
\end{tikzpicture}
\caption{The blow-up $\rho$ of $Z$ along $P$.}
\label{fig:bl}
\end{figure}
For a point $p\in\tl{H}^\sim\cup E_P=(g^\prime\circ\rho)^{-1}(\infty)$ 
by Lemmas \ref{lem:2-99} and \ref{lem:2-999} we obtain 
\begin{align}
\zeta_{h\circ g^\prime\circ\rho,p}(\iota^\prime_!\SF)(t) =
\begin{cases}
\ 1-te^{-6\pi\sqrt{-1}\alpha} 
& (p\in\tl{H}^\sim\bs((\var{f^{-1}(0)})^\approx\cup E_P\cup D_2)), \\
\ 1-te^{-2\pi\sqrt{-1}\alpha} 
& (p\in E_P\bs((\var{f^{-1}(0)})^\approx\cup \tl{H}^\sim)), \\
\ 1 & (\textit{otherwise}).
\end{cases}
\end{align} 
Moreover we have 
\begin{equation}
\chi\(\tl{H}^\sim\bs((\var{f^{-1}(0)})
^\approx\cup E_P\cup D_2)\)
=\chi\(E_P\bs((\var{f^{-1}(0)})^\approx\cup \tl{H}^\sim)\)=-1
\end{equation}
and hence 
\begin{equation}
\zeta_{h,\infty}(i_!\rmR g_!\SF)(t)=
(1-te^{-2\pi\sqrt{-1}\alpha})^{-1}(1-te^{-6\pi\sqrt{-1}\alpha})^{-1}.
\end{equation}
By Theorem \ref{MT-2}, for any non-zero complex number $\lambda\in\CC^\ast$ such that $\lambda\neq1$ we have
\begin{align}
\mu(\SM^\wedge,\lambda) = 
\begin{cases}
\ 1 & ( \lambda = e^{{-2\pi\sqrt{-1}\alpha}}, e^{{-6\pi\sqrt{-1}\alpha}}), \\
& \\
\ 0 & (\textit{otherwise}).
\end{cases}
\end{align}        
\item [\rm (ii)] 
For $N=2$, $f(z_1,z_2)=z_1^3+z_2\in\CC[z_1,z_2]$ and $\alpha\in\CC$, 
let us consider the regular meromorphic connection $\SM\in\Modrh(\SD_X)$ defined as in (i).
In this case, the hypersurface $\{f_d=0\}$ is not smooth 
and the covering degree of 
$\Char(\SM)\cap q^{-1}(\Omega)\longrightarrow\Omega$ is 2.
Similarly to the case (i), we can see that 
\begin{align}
\zeta_{h,\infty}(i_!\rmR g_!\SF)(t)&=
(1-te^{-2\pi\sqrt{-1}\alpha})(1-t^3e^{-6\pi\sqrt{-1}\alpha})^{-1} \\
&=(1+te^{-2\pi\sqrt{-1}\alpha}+t^2e^{-4\pi\sqrt{-1}\alpha})^{-1}.
\end{align}  
It follows from Theorem \ref{MT-2} that for any non-zero complex number $\lambda\in\CC^\ast$ such that $\lambda\neq1$ we have
\begin{align}
\mu(\SM^\wedge,\lambda) = 
\begin{cases}
\ 1 & ( \lambda = \beta_1, \beta_2), \\
& \\
\ 0 & (\textit{otherwise}),
\end{cases}
\end{align}        
where $\beta_1,\beta_2\in\CC^\ast$ are the non-zero complex numbers such that $\beta_1+\beta_2=-e^{{-2\pi\sqrt{-1}\alpha}}$ and $\beta_1\beta_2=e^{{-4\pi\sqrt{-1}\alpha}}$.
\end{enumerate}
\end{example}

\end{document}